\documentclass[a4paper,12pt, leqno]{article}

\usepackage[utf8]{inputenc}
\usepackage{graphicx}
\usepackage{xcolor}
\usepackage{tikz}
\tikzset{
    use page relative coordinates/.style={
        shift={(current page.south west)},
        x={(current page.south east)},
        y={(current page.north west)}
    },
}
\usetikzlibrary{decorations.pathreplacing,angles,quotes}

\usepackage{amsmath, amssymb, amsthm, amstext, bm, mathtools}
\usepackage{esint}
\usepackage{physics}
\usepackage{bigints}

\usepackage{etoolbox}
\usepackage{comment}
\usepackage{authblk}
\usepackage{arydshln}
\usepackage{enumerate}

\usepackage{biblatex}
\renewbibmacro{in:}{}
\DeclareFieldFormat[article]{journaltitle}{\mkbibemph{\textup{#1}}}
\DeclareFieldFormat[article]{title}{#1}

\usepackage{hyperref}
\hypersetup{
    pdfborder={0 0 1},
    linktocpage=true
}

\theoremstyle{plain}
\newtheorem{theorem}{Theorem}[section]
\newtheorem{lemma}[theorem]{Lemma}

\newtheorem{proposition}[theorem]{Proposition}

\newtheorem{remark}[theorem]{Remark}
\newtheorem{assumption}[theorem]{Assumption}
\numberwithin{equation}{section}

\usepackage{cleveref}



\title{Doubling Argument of the Hessian Estimate for the Hessian Quotient Equations}
\author{Cheuk Yan Fung}
\date{June 2026}

\begin{document}
\maketitle
\begin{abstract}
In this paper, we establish a doubling argument to obtain Hessian estimates for convex solutions to the Hessian quotient equation $\frac{\sigma_n}{\sigma_k}(D^2u) = f(x,u,Du)$ for $k=n-1$ and $k=n-2$ under the condition that $\log f$ is convex in the $Du$ variable. In particular, our approach is pointwise and does not make use of the Legendre transform or integral-based local maximum principles. We provide a counterexample demonstrating that interior estimates can fail if no structural assumption is imposed on $f$ in the $Du$ variable. Finally, we extend our doubling argument to general Hessian quotient equations $\frac{\sigma_l}{\sigma_k}(D^2u) = f(x,u,Du)$ for $k \in \{l-1, l-2\}$, under a similar structural condition imposed on $f$ in the $Du$ variable, alongside an additional structural concavity assumption on the operator introduced in \cite{LuTsai_general_quotient}, which has very recently been established in independent works.
\end{abstract}

\section{Introduction}
\label{s1}

Bombieri–De Giorgi–Miranda \cite{Bombieri1969} established an a priori gradient estimate for graphical minimal hypersurfaces, which was later simplified by Trudinger \cite{Trudinger1972}. The key ingredients were the Jacobi inequality $\Delta a \geq |\nabla a|^2$ and the Michael-Simon mean value inequality \cite{MichaelSimon1973}. Korevaar \cite{Korevaar1987} subsequently established a pointwise argument for this gradient estimate without relying on the Michael-Simon inequality. 

The Hessian estimate for the special Lagrangian equation can be viewed as a particular case of the gradient estimate for minimal surface systems, when the minimal surface is represented as a gradient graph. Chen-Warren-Yuan~\cite{Chen2009}, Wang-Yuan~\cite{WangYu2014SLE}, and Zhou~\cite{Zhou2021} established integral proofs for the Hessian estimates of the special Lagrangian equation, relying on the Michael-Simon mean value inequality and the Jacobi inequality. Recently, Shankar~\cite{Shankar2024SLE} and Fung \cite{fung} established pointwise, doubling arguments to obtain Hessian estimates for the special Lagrangian equation in the settings considered by Chen-Warren-Yuan~\cite{Chen2009}, Wang-Yuan~\cite{WangYu2014SLE}, and Zhou~\cite{Zhou2021}. This approach is entirely pointwise and bypasses the Michael-Simon mean value inequality required in prior proofs. 

Several works have constructed doubling inequalities to establish Hessian estimates and interior regularity results. For the $\sigma_2$ equation in dimension $n=3$, Qiu \cite{Qiu2024} derived the Hessian estimate. A major breakthrough in dimension $n=4$ was achieved by Shankar-Yuan \cite{ShankarYuansigma2dim4}. They established the Hessian estimate for $\sigma_2=1$ by pioneering the framework for the doubling argument, which is now widely adapted to other equations. The estimate also holds in $n\ge 5$ for solutions satisfying a dynamic semi-convexity condition. Subsequently, Fan \cite{Fan2026Savin} established the Hessian estimate with a variable right hand side $\sigma_2=f(x,u,Du)$ when $n=4$ by constructing a generalized Savin theorem \cite{fan2026hessian}. Furthermore, Shankar-Yuan \cite{ShankarYuan2024} proved the interior regularity of strictly convex solutions to the Monge-Amp\`ere equation. As mentioned previously, Shankar \cite{Shankar2024SLE} and Fung \cite{fung} constructed doubling arguments for the special Lagrangian equation in various settings. Additionally, Bhattacharya-Shankar-Wall \cite{BSW25Lagrangian_Largangian_mean_curvature} established the Hessian estimate for the two-dimensional Lagrangian mean curvature equation. Recently, Jiao-Sui \cite{JiaoSui2026} utilized the Lagrange multiplier method to establish doubling inequalities and Hessian estimates for $\frac{\sigma_2}{\sigma_1}(D^2u) = f(x,u)$ for 2-convex solutions in dimension $3$ and for semi-convex solutions when $n \ge 4$. The method also applies to $\sigma_2(D^2u) = f(x,u,Du)$ in dimension $4$. Very recently, Fan-Shankar \cite{FanShankar_scalar_curvature_equation} established interior curvature estimates for the scalar curvature
equation in dimension 4.

Recently, making use of the concavity inequalities established by Guan and Sroka \cite{GuanSroka2026}, Lu \cite{Lu2025_top_quotient} established interior estimates for $\frac{\sigma_n}{\sigma_k}(D^2u) = f(x,u)$ for $k \in \{n-1,n-2\}$ with $k\ge 1$. He also constructed explicit counterexamples demonstrating that such interior Hessian estimates fail when $k \le n-3$. Beyond the top quotient case, Lu and Tsai \cite{LuTsai_general_quotient} recently established the Hessian estimate for convex solutions to the general Hessian quotient equation $\frac{\sigma_l}{\sigma_k}(D^2u) = f(x,u)$ for $k= l-1$ or $l-2$ with $k\ge 1$ under an analogous concavity assumption on the operator $\frac{\sigma_l}{\sigma_k}(D^2u)$, motivated by the concavity inequalities of Guan and Sroka \cite{GuanSroka2026}. 
\begin{assumption}
[\cite{LuTsai_general_quotient} Assumption 1.1]
\label{assumption1.1}
Let $n \ge 2$, $1 \le k < l \le n$ with $k = l-1$ or $ l-2$, and let $F = \frac{\sigma_l}{\sigma_k}$. Suppose $\lambda_1 \ge \dots \ge \lambda_n > 0$. Then for any $\xi \in \mathbb{R}^n$, there exist constants $K_0 \ge 1$ and $K_1 \ge 1$ depending only on $n$, $l$, and $F$, such that whenever $\lambda_1 \ge K_0$, we have
\begin{equation*}
-\sum_{i,j}F^{ii,jj}\xi_i\xi_j + \sum_{\lambda_i \ne \lambda_1}\frac{2F^{ii}\xi_i^2}{\lambda_1-\lambda_i} + \frac{K_1}{F}\left(\sum_i F^{ii}\xi_i\right)^2 \ge (1+c(n,l))\frac{F^{11}\xi_1^2}{\lambda_1},
\end{equation*}
where $c(n,l)$ is a positive constant depending only on $n$ and $l$.
\end{assumption}

\begin{remark}
\label{rem:assumption_verification}
After the completion of this work, the concavity inequality in Assumption \ref{assumption1.1} has been independently established by Tsai \cite{Tsai2026concavity} for the case $k = l-1$ for convex solutions (their result shows that $K_1$ can be chosen as $1$), Li and Wu \cite{LiWu2026concavity} for both $k = l-1$ and $k = l-2$ for convex solutions, and Dong and Zhang \cite{DongZhang2026concavity} for $k = l-1$ and $k = l-2$ for admissible semiconvex solutions. 

Although Assumption \ref{assumption1.1} can now be removed in view of these results, we retain it here to keep the presentation of the proof unchanged.
\end{remark}

Lu \cite{Lu2025_top_quotient} and Lu-Tsai \cite{LuTsai_general_quotient} constructed Jacobi inequalities in the viscosity sense. Then they applied the Legendre transform to the solution, transforming the operator into a uniformly elliptic one, which allows the use of integral tools such as the local maximum principle. They then utilized integral techniques to establish the estimates.

A natural question is whether the pointwise, doubling approach is applicable to this case. A useful ingredient for the doubling argument, the Jacobi inequality, is available in this setting; indeed, Shankar \cite{Shankar2025Singularities} noted that possessing a Jacobi inequality usually implies a doubling inequality can be deduced. In this paper, we answer this question affirmatively. By adapting the pointwise doubling method, we establish Hessian estimates for the Hessian quotient equation without relying on the Legendre transform or integral estimates. Specifically, for the top quotient case, we consider the equation of the form
\begin{equation}
\label{main_equation_1}
F(D^2u)= \frac{\sigma_n}{\sigma_k}(D^2u) = f(x,u,Du)
\end{equation}
for $k=n-1,n-2$, under the structural condition that $\log f$ is convex with respect to the gradient variable $Du$.

Let us briefly review the outline for the doubling argument pioneered by Shankar and Yuan \cite{ShankarYuansigma2dim4}. The first step is to prove an Alexandrov-type theorem. Next, the uniform norm of the solution is flattened so that Savin's small perturbation theorem (\cite{Savin2007}, Theorem 1.3) can be applied to obtain Hessian estimates in a small ball. Finally, by employing the Jacobi inequality, one can derive the doubling inequality, which allows the Hessian estimate in a larger ball to be controlled by the estimate in the smaller ball.

In this paper, we adapt their framework to the Hessian quotient equation. We first establish the Jacobi inequality in the classical sense. Next, we establish a doubling inequality. Finally, we apply the generalized Savin theorem established by Fan \cite{Fan2026Savin} to conclude the proof. The main contributions of this paper are the classical Jacobi inequality and the doubling inequality for the Hessian quotient equation. 

We modify the arguments of Lu \cite{Lu2025_top_quotient} to establish the Jacobi inequality in the classical sense, which will be used in the maximum principle argument in Section \ref{s4}. Establishing this inequality in the classical setting introduces additional third order terms that naturally vanish in the viscosity setting. To absorb these terms when they have negative coefficients,  we must proceed case by case based on the multiplicity $m$, which in turn requires the maximal eigenvalue $\lambda$ to be sufficiently large.

We establish the doubling inequality for the Hessian quotient equation using the test function constructed by Shankar \cite{Shankar2024SLE}. By evaluating this test function at its maximum point, we obtain a maximum principle inequality \eqref{max_prin_inequality}. For the subcase $k=n-1$, we make use of the fact that the trace of the linearized operator, $\sum_i F^{ii}$, is uniformly bounded from above to obtain the desired estimate. However, in the subcase $k=n-2$, the trace is no longer bounded from above. To overcome this, we bound the right hand side of the inequality from below by a multiple of $|x-y|^2 \sum_i F^{ii}$. By dividing the entire inequality by the trace and applying a lower bound on the trace, we obtain the doubling inequality.

\begin{theorem}
\label{main_thm}
Let $n \ge 2$, $k \in \{n-1, n-2\}$ with $k \ge 1$, and let $f(x,u,Du)$ be a positive $C^{1,1}$ function such that $\log f$ is convex in the $Du$ variable. Suppose $u \in C^4(B_{2}(0))$ is a convex solution of \eqref{main_equation_1} on $B_2(0)$. Then we have 
\begin{equation}
|D^2u(0)| \le C\left(n,k,\|u\|_{C^{0,1}(B_2(0))}, \left\|\frac{1}{f}\right\|_{L^{\infty}}, \|f\|_{C^{1,1}}\right).
\end{equation}
\end{theorem}
We remark that Guan-Qiu \cite{GuanQiu2019} (under the assumption $\sigma_3(D^2u) \ge -A$) and Fan \cite{fan2026hessian} established the interior Hessian estimate for $\sigma_2(D^2u) = f(x,u,Du)$ without imposing structural conditions on $f$ with respect to the gradient variable $Du$. This is due to the properties of the $\sigma_2$ operator. Specifically, by applying the Newton-Maclaurin inequalities, we have $\sum_i \sigma_2^{ii}u_{ii}^2 = \sigma_1 \sigma_2 - 3 \sigma_3 \ge C(n) \sigma_1 \sigma_2 \ge C \sigma_1$. This term can be used to absorb the negative $\lambda$ terms. In our case, taking $\sigma_n / \sigma_{n-1}$ as an example, we have $F^{ii}u_{ii}^2 = F^2$, which is not helpful for absorbing negative $\lambda$ terms. If we do not impose the condition that $\log f$ is convex in $Du$ and instead only use the bounds on the first and second partial derivatives of $f$ in \eqref{jacobi_final}, it would generate negative $\lambda$ terms that cannot be absorbed. 

We demonstrate that imposing structural conditions on $f$ with respect to the gradient variable $Du$ is necessary by constructing the following explicit counterexample. We first consider the simplest case, $\sigma_n/\sigma_{n-1}(D^2u) = f(Du)$. We consider the dual function $w(y)$ obtained via the Legendre transform of $u(x)$. To ensure that the second derivatives of $u$ blow up at the origin, we require the Hessian of $w$ to degenerate at the origin. Under the Legendre transform, the equation simplifies to $\Delta w(y) = 1/f(y)$, where $y = Du$. Motivated by an earlier, more restrictive structural assumption in Theorem \ref{main_thm} that $1/f$ is concave in $Du$, we construct a specific dual function $w$ whose Laplacian is strictly convex in $y$. We verify that this function $w$ forms a valid counterexample and explicitly solve for $u(x)$. Finally, we demonstrate that this exact construction also serves as a valid counterexample for the general quotient $\sigma_l/\sigma_k(D^2u) = f(Du)$ for $k \in \{l-1,l-2\}$.

\begin{remark}
\label{thm:counterexample}
Let $n \ge 2$, $1 \le k < l \le n$ with $k \in \{l-1, l-2\}$. Then there exists a sequence of smooth, convex functions $u_\varepsilon \in C^\infty(B_1(0))$ and positive smooth functions $f_\varepsilon(Du)$ such that 
\begin{equation}
\label{counter_eq}
    \frac{\sigma_l}{\sigma_{k}}(D^2u_\varepsilon) = f_\varepsilon(Du_\varepsilon) \quad \text{in } B_1(0),
\end{equation}
with the following properties as $\varepsilon \to 0$.
\begin{enumerate}[(i)]
    \item The sequence $u_\varepsilon$ is uniformly bounded in $C^{0,1}(B_1(0))$.
    \item The functions $f_\varepsilon$ are uniformly bounded in $C^{1,1}$ with respect to $Du$, and there exist uniform constants such that $0 < \inf_{\varepsilon,Du} f_\varepsilon(Du) \le f_\varepsilon \le C$.
    \item The second derivatives blow up at the origin $|D^2u_\varepsilon(0)| \to \infty$.
\end{enumerate}
\end{remark}

We also establish a doubling argument for strictly convex solutions to the general Hessian quotient equation
\begin{equation}
\label{main_equation_2}
F(D^2u) = \frac{\sigma_l}{\sigma_k}(D^2u) = f(x,u,Du),
\end{equation}
where $k = l-1$ or $l-2$, under Assumption \ref{assumption1.1} established by Lu and Tsai \cite{LuTsai_general_quotient} and a structural inequality imposed on $f$ in the $Du$ variable.

We face additional difficulties in deriving the Jacobi inequality compared to the top quotient case. The coefficients of certain third order terms are complicated, and we introduce auxiliary quantities to simplify them. Furthermore, the coefficients of certain third order terms depend on potentially unbounded eigenvalues $\lambda_{m+1}, \dots, \lambda_{l-1}$. We cannot take $\lambda \rightarrow \infty$ at the beginning because the choice of $\lambda$ might depend on these unbounded terms. We overcome this by factoring out those eigenvalues. 

For the doubling inequality in the general quotient case, we rely on the estimates established by Lu and Tsai \cite{LuTsai_general_quotient}. Specifically, we utilize the property that the trace $\sum_i F^{ii} \le C$ for the subcase $k=l-1$, and that the quantity $F^{ii}(1+u_{ii})^2$ is bounded from both above and below by a constant multiple of $\lambda_{l-1}$ for the subcase $k=l-2$. By combining these bounds with the approach developed for the top quotient, we derive the desired doubling inequality.

\begin{theorem}
\label{main_thm_general_l_k}
Let $n \ge 2$, $1 \le k < l \le n$ with $k \in \{l-1, l-2\}$, and let $f(x,u,Du)$ be a positive $C^{1,1}$ function satisfying the structural inequality $\sum_{i,j} f_{p_i p_j} \xi_i \xi_j \ge \frac{K_1}{f} \left(\sum_i f_{p_i} \xi_i\right)^2$ for all $\xi \in \mathbb{R}^n$. Suppose Assumption \ref{assumption1.1} holds and $u \in C^4(B_{2}(0))$ is a strictly convex solution to \eqref{main_equation_2} on $B_2(0)$. Then we have 
\begin{equation}
|D^2u(0)| \le C\left(n, l, \|u\|_{C^{0,1}(B_2(0))}, \left\|\frac{1}{f}\right\|_{L^{\infty}}, \|f\|_{C^{1,1}}\right).
\end{equation}
\end{theorem}

We also review recent progress on interior Hessian estimates for Hessian quotient equations. In dimensions $n = 3$ and $4$, estimates for $\frac{\sigma_3}{\sigma_1}(D^2u) = 1$ were originally obtained by Chen–Warren–Yuan \cite{Chen2009} and Wang–Yuan \cite{WangYu2014SLE} as a consequence of their work on the special Lagrangian equation. Zhou \cite{Zhou2024} subsequently generalized these estimates to the variable coefficient case $\frac{\sigma_3}{\sigma_1}(D^2u) = f(x)$, again as a consequence of his study on the twisted special Lagrangian equation. More recently, Mei–Yan \cite{MeiYan2026} obtained Hessian estimates for semi-convex solutions to $\frac{\sigma_3}{\sigma_l}(D^2u) = 1$ ($l \in \{1,2\}$) in general dimensions. Additionally, Lu and Sroka \cite{LuSroka2026} established a Liouville theorem for semi-convex entire solutions to $\frac{\sigma_2}{\sigma_1}(D^2u) = 1$ by reformulating the quotient operator as a $\sigma_2$ operator. Finally, for $\frac{\sigma_2}{\sigma_1}(D^2u) = f(x,u)$, Jiao and Sui \cite{JiaoSui2026} obtained Hessian estimates for $2$-convex solutions when $n=3$, and for semi-convex solutions when $n \ge 4$.

The structure of the paper is as follows. In Section \ref{s2}, we collect some established properties of the Hessian quotient operators, establish several additional preliminary results, and restate the generalized Savin small perturbation theorem. In Section \ref{s3}, we prove the classical Jacobi inequalities. In Section \ref{s4}, we prove the doubling inequalities. In Section \ref{s5}, we complete the proof of the Hessian estimates (Theorem \ref{main_thm} and Theorem \ref{main_thm_general_l_k}). Finally, in Section \ref{s6}, we construct a counterexample (Remark \ref{thm:counterexample}).

\section{Preliminaries}
\label{s2}
In this section, we restate several established properties of the Hessian quotient operators and prove several additional properties. We also restate the generalized Savin small perturbation theorem.
Throughout the article, let $\lambda_1\geq \lambda_2 \ge \dots \ge \lambda_n > 0$ be the eigenvalues of $D^2u$.

We first restate the formulas for the first and second derivatives of the operator $F$ with respect to the matrix entries.
\begin{lemma} [\cite{LuTsai_general_quotient} Lemma 2.2]\label{lem:2.2}
Let $n \ge 2$, $1 \le k < l \le n$, and $F = \frac{\sigma_l}{\sigma_k}(D^2u)$. Suppose that $u$ is a $C^2$ strictly convex function and $D^2u$ is diagonalized at $x_0$. Then at $x_0$, we have
\begin{align*}
F^{pq} &= \left(\frac{\sigma_{l}^{pq}}{\sigma_{k}} - \frac{\sigma_{l}\sigma_{k}^{pq}}{\sigma_{k}^{2}}\right)\delta_{pq}, \\
F^{pq,rs} &=
\begin{cases}
-2\frac{\sigma_{l}^{pp}\sigma_{k}^{pp}}{\sigma_{k}^{2}} + 2\frac{\sigma_{l}(\sigma_{k}^{pp})^{2}}{\sigma_{k}^{3}}, & p=q=r=s, \\
\frac{\sigma_{l}^{pp,qq}}{\sigma_{k}} - \frac{\sigma_{l}^{pp}\sigma_{k}^{qq} + \sigma_{l}^{qq}\sigma_{k}^{pp}}{\sigma_{k}^{2}} - \frac{\sigma_{l}\sigma_{k}^{pp,qq}}{\sigma_{k}^{2}} + 2\frac{\sigma_{l}\sigma_{k}^{pp}\sigma_{k}^{qq}}{\sigma_{k}^{3}}, & p=q, r=s, p \ne r, \\
-\frac{\sigma_{l}^{pp,qq}}{\sigma_{k}} + \frac{\sigma_{l}\sigma_{k}^{pp,qq}}{\sigma_{k}^{2}}, & p=s, q=r, p \ne q, \\
0, & \mathrm{otherwise}.
\end{cases}
\end{align*}
In particular, for all $p \ne q$, we have $-F^{pq,qp}>0$ and
\begin{equation*}
-F^{pq,qp} = \frac{F^{pp} - F^{qq}}{\lambda_{q} - \lambda_{p}}, \quad \lambda_{p} \ne \lambda_{q},
\end{equation*}
where $\lambda_{p}$ and $\lambda_{q}$ are the eigenvalues of $D^2u$.
\end{lemma}

For completeness, we provide the proof that $-F^{pq,qp} \ge 0$ on the positive cone. The proof for the case $l=n$ is given in \cite{Lu2025_top_quotient} Lemma 2.2.

\begin{proof}
For $p \neq q$,
\begin{equation*}
  -F^{pq,qp} = \frac{1}{\sigma_k^2} (\sigma_l^{pp,qq}\sigma_k - \sigma_l\sigma_k^{pp,qq} ).  
\end{equation*}

Let $T_m = \sigma_m(\lambda | \lambda_p, \lambda_q)$. We can expand $\sigma_l$ and $\sigma_k$ as
\begin{equation*}
    \sigma_m = \lambda_p\lambda_q T_{m-2} + (\lambda_p+\lambda_q)T_{m-1} + T_m.
\end{equation*}
\begin{align*}
    -F^{pq,qp}\sigma_k^2 &=  T_{l-2} \big( \lambda_p\lambda_q T_{k-2} + (\lambda_p+\lambda_q)T_{k-1} + T_k \big) \\
    &- \big( \lambda_p\lambda_q T_{l-2} + (\lambda_p+\lambda_q)T_{l-1} + T_l \big) T_{k-2}\\
    &=(\lambda_p+\lambda_q) \left[ T_{l-2}T_{k-1} - T_{l-1}T_{k-2} \right] + \left[ T_{l-2}T_k - T_l T_{k-2} \right] > 0,
\end{align*}
where we use Newton's inequality and $\lambda_p,\lambda_q > 0$ in the last inequality.
\end{proof}

We restate the explicit formulas for the first derivatives of the operator $F(D^2u) = \frac{\sigma_n}{\sigma_{k}}$, where $k \in \{n-1,n-2\}$ with $k\ge 1$. They are taken from the proofs of Lemma 3.1 and Lemma 3.2 in \cite{Lu2025_top_quotient}.
\begin{lemma}
[\cite{Lu2025_top_quotient}]
\label{n-1properties}
Let $n \ge 2$ and $F(\lambda) = \frac{\sigma_{n}}{\sigma_{n-1}}(\lambda)$ for $\lambda \in \mathbb{R}^n$ with $\lambda_{1} \ge \dots \ge \lambda_{n} > 0$. We may rewrite $F(\lambda)$ as
\begin{align*}
F(\lambda) &= \frac{1}{\sum_i \frac{1}{\lambda_i}}.
\end{align*}
Taking derivatives with respect to $\lambda_i$, we have
\begin{align*}
F^{ii} &= \frac{1}{\left( \sum_k \frac{1}{\lambda_k} \right)^2} \frac{1}{\lambda_i^2} = \frac{F^2}{\lambda_i^2}.
\end{align*}
\end{lemma}

\begin{lemma}
[\cite{Lu2025_top_quotient}]
\label{n-2properties}
Let $n \ge 2$ and $F(\lambda) = \frac{\sigma_{n}}{\sigma_{n-2}}(\lambda)$ for $\lambda \in \mathbb{R}^n$ with $\lambda_{1} \ge \dots \ge \lambda_{n} > 0$. We may rewrite $F(\lambda)$ as
\begin{align*}
F(\lambda) &= \frac{\lambda_1 \lambda_2 \dots \lambda_n}{\sum_{i < j} \lambda_1 \dots \hat{\lambda}_i \dots \hat{\lambda}_j \dots \lambda_n} = \frac{1}{\sum_{i<j} \frac{1}{\lambda_i \lambda_j}}.
\end{align*}
Taking derivatives with respect to $\lambda_a$, we have
\begin{align*}
F^{aa} &= \frac{1}{\left( \sum_{i<j} \frac{1}{\lambda_i \lambda_j} \right)^2} \frac{1}{\lambda_a^2} \sum_{b \neq a} \frac{1}{\lambda_b} = F^2 \frac{1}{\lambda_a^2} \sum_{b \neq a} \frac{1}{\lambda_b}.
\end{align*}
\end{lemma}
We state a formula relating $-F^{ab,ba}$ and $\frac{F^{aa}}{\lambda_a}$ that will be used in the proof of the Jacobi inequalities.
\begin{lemma}
\label{lem:F_abba_exact}
Let $n \ge 2$ and $F(D^2u) = \frac{\sigma_n}{\sigma_k}(D^2u)$ for $k \in \{n-1, n-2\}$ with $k\ge 1$. Let the eigenvalues be $\lambda_1 \ge \dots \ge \lambda_n > 0$ at $x_0$. Suppose there exist distinct indices $a \neq b$ such that $\lambda_a = \lambda_b = \lambda$ at $x_0$. Then $-F^{ab,ba}$ evaluated at $x_0$ satisfies the following.
\begin{enumerate}[(a)]
    \item For $k = n-1$,
    \begin{equation*}
    -F^{ab,ba} =  \frac{2F^{aa}}{\lambda}.
    \end{equation*}
    \item For $k = n-2$,
    \begin{equation*}
    -F^{ab,ba} = \frac{2F^{aa}}{\lambda} - \frac{F^2}{\lambda^4}.
    \end{equation*}
\end{enumerate}
\end{lemma}

\begin{proof}
By Lemma \ref{lem:2.2}, we have
\begin{equation}
\label{exact_cross_derivative}
-F^{ab,ba} = \frac{\sigma_n^{aa,bb}}{\sigma_k} - \frac{\sigma_n \sigma_k^{aa,bb}}{\sigma_k^2} = \frac{F}{\lambda_a \lambda_b} - F \frac{\sigma_k^{aa,bb}}{\sigma_k}.
\end{equation}
(a) For $k=n-1$, we have
\begin{align*}
    \sigma_{n-1}^{aa,bb} &= \sigma_{n-3}(\lambda| \lambda_a,\lambda_b) =   \sigma_{n-2}(\lambda| \lambda_a,\lambda_b) \sum_{i \neq a,b}\frac{1}{\lambda_i} = \frac{\sigma_n}{\lambda_a \lambda_b} \sum_{i \neq a,b} \frac{1}{\lambda_i} \\
    &= \frac{\sigma_n}{\lambda_a \lambda_b} \left( \frac{1}{F} - \frac{1}{\lambda_a} - \frac{1}{\lambda_b} \right),
\end{align*}
where we use $\frac{1}{F} = \sum_i \frac{1}{\lambda_i}$ in the last equality.
Substituting this into (\ref{exact_cross_derivative}) and evaluating it at $\lambda_a = \lambda_b = \lambda$ gives
\begin{align*}
    -F^{ab,ba} &= \frac{F}{\lambda_a \lambda_b} - \frac{F^2}{\lambda_a \lambda_b} \left( \frac{1}{F} - \frac{1}{\lambda_a} - \frac{1}{\lambda_b} \right) \\
    &= \frac{F^2(\lambda_a + \lambda_b)}{\lambda_a^2 \lambda_b^2} = \frac{2F^2}{\lambda^3} = \frac{2F^{aa}}{\lambda}.
\end{align*}
(b) For $k=n-2$, 
\begin{align*}
    \sigma_{n-2}^{aa,bb}&= \sigma_{n-4}(\lambda| \lambda_a,\lambda_b) =   \sigma_{n-2}(\lambda| \lambda_a,\lambda_b)   \sum_{\substack{c < d \\ c,d \neq a,b}} \frac{1}{\lambda_c \lambda_d} \\
    &=  \frac{\sigma_n}{\lambda_a \lambda_b} \sum_{\substack{c < d \\ c,d \neq a,b}} \frac{1}{\lambda_c \lambda_d} \\
    &= \frac{\sigma_n}{\lambda_a \lambda_b} 
    \left( \frac{1}{F} - \frac{1}{\lambda_a \lambda_b} - \left(\frac{1}{\lambda_a} + \frac{1}{\lambda_b}\right) \sum_{c \neq a,b} \frac{1}{\lambda_c} \right),
\end{align*}
where we use $\frac{1}{F} = \sum_{i < j} \frac{1}{\lambda_i \lambda_j}$ in the last equality.

Substituting this into (\ref{exact_cross_derivative}) and evaluating it at $\lambda_a = \lambda_b = \lambda$ gives
\begin{align*}
    -F^{ab,ba} &= \frac{F}{\lambda_a \lambda_b} - \frac{F^2}{\lambda_a \lambda_b} \left[ \frac{1}{F} - \frac{1}{\lambda_a \lambda_b} - \left(\frac{1}{\lambda_a} + \frac{1}{\lambda_b}\right) \sum_{c \neq a,b} \frac{1}{\lambda_c} \right] \\
    &= \frac{F^2}{\lambda_a^2 \lambda_b^2} + \frac{F^2(\lambda_a + \lambda_b)}{\lambda_a^2 \lambda_b^2} \sum_{c \neq a,b} \frac{1}{\lambda_c}\\
    & =\frac{F^2}{\lambda^4} + \frac{2F^2}{\lambda^3} \sum_{c \neq a,b} \frac{1}{\lambda_c}.
\end{align*}

By Lemma \ref{n-2properties}, 
\begin{equation}
    F^{aa} = \frac{F^2}{\lambda^3} + \frac{F^2}{\lambda^2} \sum_{c \neq a,b} \frac{1}{\lambda_c}.
\end{equation}
We have the desired equality.
\end{proof}

The following equality will be used in the Jacobi inequality for the general $l,k$ case.
\begin{lemma}
\label{lem:F_abba_general_l_k}
Let $n \ge 2$, $1 \le k < l \le n$, and let $F(D^2u) = \frac{\sigma_l}{\sigma_{k}}(D^2u)$. Let the eigenvalues be $\lambda_1 \ge \dots \ge \lambda_n > 0$ at $x_0$. Suppose there exist distinct indices $a \neq b$ such that $\lambda_a = \lambda_b = \lambda$ at $x_0$. Then $-F^{ab,ba}$ evaluated at $x_0$ satisfies the following equality. 
\begin{equation}
\label{golden_eq}
F^{aa} = \lambda_b (-F^{ab,ba}) + G^{ab},
\end{equation}
where $G^{ab}$ is defined by
\begin{equation}
\label{G_def}
G^{ab} := \frac{\sigma_{l-1}(\lambda|\lambda_a, \lambda_b)}{\sigma_{k}} - F \frac{\sigma_{k-1}(\lambda|\lambda_a, \lambda_b)}{\sigma_{k}}.
\end{equation}
\end{lemma}
\begin{proof}
By Lemma \ref{lem:2.2}, we have
\begin{equation*}
-F^{ab,ba} = \frac{\sigma_{l-2}(\lambda|\lambda_a, \lambda_b)}{\sigma_{k}} - F \frac{\sigma_{k-2}(\lambda|\lambda_a, \lambda_b)}{\sigma_{k}}.
\end{equation*}
The first derivative with respect to $\lambda_a$ is
\begin{equation}
\label{Faa_l}
F^{aa} = \frac{\sigma_{l-1}(\lambda|\lambda_a)}{\sigma_{k}} - F \frac{\sigma_{k-1}(\lambda|\lambda_a)}{\sigma_{k}}.
\end{equation}
Using the identity $\sigma_{j}(\lambda|\lambda_a) = \lambda_b \sigma_{j-1}(\lambda|\lambda_a, \lambda_b) + \sigma_{j}(\lambda|\lambda_a, \lambda_b)$, we have
\begin{align*}
F^{aa} &= \lambda_b \left[ \frac{\sigma_{l-2}(\lambda|\lambda_a, \lambda_b)}{\sigma_{k}} - F \frac{\sigma_{k-2}(\lambda|\lambda_a, \lambda_b)}{\sigma_{k}} \right] \\
&+ \left[ \frac{\sigma_{l-1}(\lambda|\lambda_a, \lambda_b)}{\sigma_{k}} - F \frac{\sigma_{k-1}(\lambda|\lambda_a, \lambda_b)}{\sigma_{k}} \right].
\end{align*}
The first bracketed term is exactly $-F^{ab,ba}$, and the second bracketed term is our definition of $G^{ab}$. This completes the proof.
\end{proof}
We establish a lower bound for the trace of the linearized operator $F^{ii}$.
\begin{lemma}
\label{general_trace_lower_bound}
Let $n \ge 3$, $l \le n$, and $k=l-2$ with $k\ge 1$. Suppose $\lambda_1 \ge \dots \ge \lambda_n > 0$ and $F(\lambda) = \frac{\sigma_l}{\sigma_k}(\lambda) = f > 0$. Then the trace of the linearized operator is bounded from below.
\begin{equation}
    \sum_i F^{ii} \ge C(n,l) \sqrt{f},
\end{equation}
where $C(n,l) = \frac{2(n-l+2)}{l-1} \sqrt{\frac{(l-1)(n-l+1)}{l(n-l+2)}}$.

\end{lemma}
\begin{proof}
By Lemma \ref{lem:2.2}, we have
\begin{equation*} 
F^{ii} = \frac{\sigma_l^{ii}}{\sigma_k} - \frac{\sigma_l \sigma_k^{ii}}{\sigma_k^2}.
\end{equation*}
Summing over $i$ and using the identity $\sum_i \frac{\partial \sigma_k}{\partial \lambda_i} = (n - k + 1) \sigma_{k-1}$, we have 
\begin{align*}
\sum_i F^{ii} &= \frac{\sum_i \sigma_l^{ii}}{\sigma_k} - \frac{\sigma_l \sum_i \sigma_k^{ii}}{\sigma_k^2} \\
&= \frac{(n-l+1)\sigma_{l-1}}{\sigma_k}-\frac{  \sigma_l (n - k + 1) \sigma_{k-1}}{\sigma_k^2} \\
&= \frac{(n-l+1)\sigma_{l-1}}{\sigma_k} - \frac{(n - k + 1)}{\sigma_k} \sigma_{k-1} F.
\end{align*}
For $k = l-2$ and $F=f$, this becomes
\begin{equation*} 
\sum_i F^{ii} = \frac{(n-l+1)\sigma_{l-1}}{\sigma_{l-2}} - \frac{(n - l + 3)}{\sigma_{l-2}} \sigma_{l-3} f.
\end{equation*}
By Newton's inequality, 
\begin{equation*} 
\left( \frac{\sigma_{l-2}}{\binom{n}{l-2}} \right)^2 \ge \frac{\sigma_{l-1}}{\binom{n}{l-1}} \frac{\sigma_{l-3}}{\binom{n}{l-3}}.
\end{equation*}
This simplifies to
\begin{equation*} 
\frac{\sigma_{l-3}}{\sigma_{l-2}} \le \frac{(l-2)(n-l+2)}{(l-1)(n-l+3)} \frac{\sigma_{l-2}}{\sigma_{l-1}}.
\end{equation*}
Substituting this into the trace equation yields
\begin{align*}
\sum_i F^{ii} 
&\ge (n-l+1)\frac{\sigma_{l-1}}{\sigma_{l-2}} - f \frac{(l-2)(n-l+2)}{l-1} \frac{\sigma_{l-2}}{\sigma_{l-1}}.
\end{align*}
Now, we minimize the function $g(x) := c_1 x - \frac{c_2f}{x}$, where $c_1 = n-l+1$, $c_2 = \frac{(l-2)(n-l+2)}{l-1}$, and $x = \frac{\sigma_{l-1}}{\sigma_{l-2}}$.
For a fixed $f$, we view $g$ as a function of $x$. Since $g'(x) = c_1 + \frac{c_2f}{x^2} > 0$, the function $g(x)$ is strictly increasing. It therefore suffices to evaluate $g$ at the minimum possible value of $x$.
By Newton's inequality,
\begin{equation*}
    f = \frac{\sigma_l}{\sigma_{l-2}} \le \frac{(l-1)(n-l+1)}{l(n-l+2)} \left(\frac{\sigma_{l-1}}{\sigma_{l-2}}\right)^2 = \frac{(l-1)(n-l+1)}{l(n-l+2)} x^2.
\end{equation*}
Thus, $x \ge \sqrt{\frac{l(n-l+2)}{(l-1)(n-l+1)}} \sqrt{f} := x_{min}$.
Hence, evaluating the trace lower bound at this minimum value $x_{min}$ gives the desired lower bound.
Let $K = \sqrt{\frac{l(n-l+2)}{(l-1)(n-l+1)}}$. Substituting the values of $c_1$, $c_2$, and $K$, we have 
\begin{align*}
    c_1 K - \frac{c_2}{K} &= \frac{1}{K} (c_1 K^2 - c_2) \\
    &= \frac{1}{K} \left( (n-l+1)\frac{l(n-l+2)}{(l-1)(n-l+1)} - \frac{(l-2)(n-l+2)}{l-1} \right) \\
    &= \frac{1}{K} \left( \frac{l(n-l+2) - (l-2)(n-l+2)}{l-1} \right) = \frac{1}{K} \frac{2(n-l+2)}{l-1}.
\end{align*}
Since $l \ge 3$ and $n \ge l$, this coefficient is positive. Therefore,
\begin{equation*}
    \sum_i F^{ii} \ge \frac{2(n-l+2)}{l-1} \sqrt{\frac{(l-1)(n-l+1)}{l(n-l+2)}} \sqrt{f}.
\end{equation*}
\end{proof}
We establish the following inequality, which will be used in the proof of the doubling inequality.
\begin{lemma}
\label{n-2_trace_upperbound}
Let $n \ge 3$. Suppose $\lambda = (\lambda_1, \dots, \lambda_n)$ satisfies $\lambda_1 \ge \dots \ge \lambda_n > 0$ and $F(\lambda) = \frac{\sigma_n}{\sigma_{n-2}}(\lambda) = f > 0$. Then the trace of the linearized operator satisfies
\begin{equation*}
\sum_i F^{ii} < f \sum_j \frac{1}{\lambda_j}.
\end{equation*}
\end{lemma}

\begin{proof}
Using the trace formula derived from the quotient rule, we have
\begin{equation*} 
\sum_i F^{ii} = \frac{\sigma_{n-1}}{\sigma_{n-2}} - 3 f \frac{\sigma_{n-3}}{\sigma_{n-2}}.
\end{equation*}

Using the relation $\sigma_n = f\sigma_{n-2}$ in the first term yields
\begin{equation*} 
\sum_i F^{ii} = f \frac{\sigma_{n-1}}{\sigma_{n}} - 3 f \frac{\sigma_{n-3}}{\sigma_{n-2}} < f \frac{\sigma_{n-1}}{\sigma_{n}} = f \sum_j \frac{1}{\lambda_j}.
\end{equation*}
\end{proof}
We restate some upper and lower bounds for the eigenvalues of the solutions of Hessian quotient equations.
\begin{lemma}[\cite{Lu2025_top_quotient} Lemma 2.3 and Lemma 2.4]
\label{lu_eigenvalue_bounds}
Let $n \ge 2$ and $ k \in \{ n-1,n-2\}$ with $k\ge 1$. Let $u \in C^4(B_{2}(0))$ be a convex solution of \eqref{main_equation_1} with $F(D^2u) = f(x,u,Du) > 0$. Then the eigenvalues of $D^2u$ satisfy
\begin{equation*}
f \leq \lambda_n \leq C(n)f \quad \mathrm{if} \quad k = n-1,
\end{equation*}
and
\begin{equation*}
\lambda_{n-1} \geq \sqrt{\frac{f}{C(n)}} > 0, \quad \lambda_n \le \sqrt{C(n)f} \quad \mathrm{if} \quad k = n-2.
\end{equation*}
\end{lemma}
\begin{remark}
In particular, all eigenvalues are positive.
For the case $k=n-1$, this is a direct consequence.
For the case $k=n-2$, suppose it were true that $\lambda_n = 0$, then $\sigma_{n-2}>0$ and $\sigma_n =0$. This is impossible because $F(D^2u) = f > 0$. 
\end{remark}

\begin{lemma} [\cite{LuTsai_general_quotient} Lemma 2.3] \label{lem:2.3}
Let $n \ge 2$, $1 < l \le n$, and let $F = \frac{\sigma_{l}}{\sigma_{l-1}}$. Suppose $\lambda_1 \ge \dots \ge \lambda_n > 0$. Then we have
\begin{equation*}
    \frac{1}{c(n,l) F} \le \lambda_l \le \frac{c(n,l)}{F},
\end{equation*}
where $c(n,l)$ is a positive constant depending only on $n$ and $l$.
\end{lemma}

\begin{lemma} [\cite{LuTsai_general_quotient} Lemma 2.4] \label{lem:2.4}
Let $n \ge 3$, $2 < l \le n$, and let $F = \frac{\sigma_{l}}{\sigma_{l-2}}$. Suppose $\lambda_1 \ge \dots \ge \lambda_n > 0$. Then we have
\begin{align*}
    \frac{1}{c(n,l) F} &\le \lambda_{l-1}\lambda_l \le \frac{c(n,l)}{F}, \\
    \lambda_{l-1} &\ge \frac{1}{c(n,l)\sqrt{F}}, \\
    \lambda_l &\le \frac{c(n,l)}{\sqrt{F}},
\end{align*}
where $c(n,l)$ is a positive constant depending only on $n$ and $l$.
\end{lemma}

We restate the concavity inequality. Taking $\tilde{\delta} = 1$, choosing $\xi_{\alpha\beta}$ to be diagonal, and restricting to $k \in \{1,2\}$ in \cite[Theorem 1.1]{GuanSroka2026}, we have the following simplified statement. 
\begin{lemma}[Guan-Sroka \cite{GuanSroka2026} Theorem 1.1]
\label{concave_simplified}
Let $n \ge 2$, $k \in \{n-1, n-2\}$ with $k \ge 1$ and let $F(\lambda) = \frac{\sigma_{n}}{\sigma_{k}}(\lambda)$ with $\lambda_{1} \ge \dots \ge \lambda_{n} > 0$. Then for any $\xi \in \mathbb{R}^{n}$, we have
\begin{equation*}
-\sum_{i,j} F^{ii,jj} \xi_{i}\xi_{j} - \frac{F^{11}\xi_{1}^{2}}{\lambda_{1}} \ge -\frac{1}{F} \left(\sum_i F^{ii}\xi_{i}\right)^{2} + M(n,k) \frac{F^{11}\xi_{1}^{2}}{\lambda_{1}},
\end{equation*}
where $M(n,k)$ is a positive constant depending on $n$ and $k$.
\end{lemma}

We restate the following generalization of Savin small perturbation theorem for equations $F(M,p,z,x)$ depending on $D^2u$, $Du$, $u$, and $x$.
\begin{theorem} [Generalization of Savin small perturbation theorem, \cite{Fan2026Savin} Theorem 1.6]
\label{generalized_savin}
Consider the fully nonlinear equation of the general form
\begin{equation*}
    F(D^{2}u, Du, u, x) = f(x) \quad \text{on } B_{1},
\end{equation*}
where $F: \mathcal{S}^{n} \times \mathbb{R}^{n} \times \mathbb{R} \times B_{1} \rightarrow \mathbb{R}$ is a function defined for $(M, p, z, x)$ satisfying the following hypotheses.

\begin{itemize}
    \item[H1)] $F(\cdot, p, z, x)$ is elliptic, i.e., $F(M+N, p, z, x) \ge F(M, p, z, x)$ for $M, N \in \mathcal{S}^{n}$ and $N \ge 0$.
    
    \item[H2)] $F(\cdot, p, z, x)$ is uniformly elliptic in a $\rho$-neighborhood of the origin in $\mathcal{S}^{n}$. That is, there exist constants $\Lambda > \lambda > 0$ such that for any $\|M\|, \|N\|, |p|, |z| \le \rho$ with $N \ge 0$, 
    $$\Lambda\|N\| \ge F(M+N, p, z, x) - F(M, p, z, x) \ge \lambda\|N\|.$$
    
    \item[H3)] $0$ is a solution, i.e., $F(0, 0, 0, x) \equiv 0$. Moreover, $F$ satisfies the following structure condition. For any $\|M\|, |p|, |q|, |z|, |s| \le \rho$ and $x \in B_{1}$, there holds 
    $$|F(M, p, z, x) - F(M, q, s, x)| \le b_{0}|p-q| + c_{0}|z-s|.$$
    
    \item[H4)] $F \in C^{1}$ and its derivative $D_{M}F$ is uniformly continuous in the $\rho$-neighborhood of $\{(0, 0, 0, x) : x \in B_{1}\}$ with a modulus of continuity $\omega_{F}$.
\end{itemize}

Let $u \in C(B_{1})$ be a viscosity solution to the equation. For any $\alpha \in (0,1)$, there exist constants $ \delta, C > 0$, depending only on $n, \alpha, \rho, \lambda, \Lambda, b_{0}, c_{0}$, and $\omega_{F}$, such that if 
$$|F(M, p, z, x) - F(M, p, z, x^{\prime})| \le \delta|x-x^{\prime}|^{\alpha}$$
for all $\|M\|, |p|, |z| \le \rho$, $x, x^{\prime} \in B_{1}$, and
$$\|u\|_{L^{\infty}(B_{1})} \le \delta, \quad \|f\|_{C^{0,\alpha}(B_{1})} \le \delta,$$
then $u \in C^{2,\alpha}(B_{1/2})$ with 
$$\|u\|_{C^{2,\alpha}(B_{1/2})} \le C.$$
\end{theorem}

\section{Jacobi Inequality}
\label{s3}
In this section, we establish the Jacobi inequalities for our equations in the classical sense. 

Our proof follows the original arguments in Lemma 4.1 of \cite{Lu2025_top_quotient} and Lemma 2.5 of \cite{LuTsai_general_quotient} with a few modifications. 
The original arguments established the Jacobi inequalities for $b_1:=\ln \lambda_1$ in the viscosity sense. 
At points where the largest eigenvalue of $D^2u$ has multiplicity $m$, their proof utilizes Lemma 5 in Brendle, Choi, and Daskalopoulos \cite{vis_lemma}, which guarantees that the third order terms $u_{abi}$ vanish for distinct indices $a \neq b \le m$.

In our proof, assuming the maximum eigenvalue of $D^2u$ has multiplicity $m$ at $x_0$, the function $b_m := \frac{1}{m}\sum_{a=1}^m \ln \lambda_a$ is locally smooth near $x_0$. We derive this inequality in the classical sense for all $1 \le m \le n$ because it will be used for the doubling inequality later in the proof. However, this classical approach encounters two obstacles. 
First, the terms $u_{abi}$ do not naturally vanish for indices $a \neq b \le m$. The terms involving $u_{abi}^2$ with negative coefficients must be carefully absorbed. While the top quotient cases $(n, n-1)$ and $(n, n-2)$ admit simpler arguments, the specific case of $(n,k) = (3,1)$ and the general quotient $\frac{\sigma_l}{\sigma_k}$ require a more delicate treatment. We must proceed case by case based on the multiplicity $m$, which in turn requires the maximal eigenvalue $\lambda$ to be sufficiently large.

In particular, for the general $l$ case, the coefficient of $u_{abi}^2$ is highly complicated. Precisely, the coefficient of $u_{abi}^2/\lambda^2$, where $\lambda$ is the maximum eigenvalue, is given by \eqref{u_aab_l-1_1}. To simplify the computation, we introduce the auxiliary quantities $G^{ab}$ (defined in \eqref{G_def}) and $T_j$ (defined in \eqref{T_j}). After simplification, we obtain the lower bound \eqref{u_aab_l-1_2}.

However, we cannot simply take the limit as $\lambda \rightarrow \infty$ in \eqref{u_aab_l-1_2}. In addition to the maximal eigenvalue $\lambda$ and the bounded eigenvalues $\lambda_l, \dots, \lambda_n$, there exist potentially unbounded eigenvalues $\lambda_{m+1}, \dots, \lambda_{l-1}$. A direct limit argument fails because the choice of $\lambda$ might depend on these terms. We overcome this difficulty by factoring out $\lambda_{m+1}, \dots, \lambda_{l-1}$ to obtain \eqref{u_aab_l-1_3n-1} and \eqref{u_aab_l-1_3n-2}. This allows us to take the limit $\lambda \rightarrow \infty$.

Second, differentiating the equation twice generates additional second order and third order terms involving the derivatives of $f$. As explained in the introduction, to eliminate the second order terms, we impose the structural condition that $\log f$ is convex with respect to $Du$. The third order terms will be handled during the proof of the doubling inequality later. For the general quotient $\frac{\sigma_l}{\sigma_k}$, we impose a similar structural condition on $f$.

Before proceeding with the computations, we clarify our derivative notation for the function $f(x, u, Du)$. We use $\frac{d}{dx_a}f$ to denote the total derivative with respect to $x_a$. 
We denote the partial derivatives of $f(x, z, p)$ by $\frac{\partial f}{\partial x_a}$, $\frac{\partial f}{\partial z}$ (or $f_z$), and $\frac{\partial f}{\partial p_i}$ (or $f_{p_i}$), corresponding to the $x$, $u$, and $Du$ variables respectively.

Additionally, we use the following index conventions for convenience. Indices $a, b, c$ usually represent indices $1, \dots, m$, where $m$ is the multiplicity of the largest eigenvalue of $D^2u$ at $x_0$. When any of $a, b$, or $c$ appears together in a term, it implies that they are mutually distinct. Indices $\alpha, \beta$ usually represent indices $m+1, \dots, n$, and when they are being used together, it implies that $\alpha \neq \beta$. Other indices like $i, j, k$ are used in the standard way to represent all indices $1, \dots, n$.
\begin{proposition}
\label{jacobi}
Let $n \ge 2$, $k \in \{n-1, n-2\}$ with $k \ge 1$, and $f > 0$ be a $C^{1,1}$ function such that $\log f$ is convex in the gradient variable $Du$. Suppose $u \in C^4(B_{2}(0))$ is a convex solution to \eqref{main_equation_1}. At any fixed point $x_0 \in B_2(0)$ where $D^2u$ is diagonalized with its eigenvalues ordered as
\begin{equation*}
    \lambda_1 = \dots = \lambda_m > \lambda_{m+1} \ge \dots \ge \lambda_n > 0,
\end{equation*} 
the function $b_m := \frac{1}{m} \sum_{a=1}^m \ln \lambda_a$ is locally smooth near $x_0$. Furthermore, there exists a constant $K_0 \ge 1$, depending only on $n, k, \|u\|_{C^{0,1}(B_2(0))}, \|1/f\|_{L^\infty}$, and $\|f\|_{C^{1,1}}$, such that if $\lambda_1 \ge K_0$, $b_m$ satisfies the Jacobi inequality 
\begin{equation*}
\sum_i F^{ii}(b_m)_{ii} \ge M(n,k)\sum_i F^{ii}((b_m)_i)^2 + \sum_i f_{p_i} (b_m)_i  - C
\end{equation*}
at $x_0$ in the classical sense. Here, $M(n,k)$ is a positive constant depending on $n$ and $k$, and $C$ is a positive constant depending only on  $n, k, \|u\|_{C^{0,1}(B_2(0))}, \|1/f\|_{L^\infty}$, and $\|f\|_{C^{1,1}}$.

Equivalently, under the same assumptions, $a_m:= \exp(M(n,k) b_m)$ satisfies 
\begin{equation}
\label{a_jacobi_1}
\sum_i F^{ii} (a_m)_{ii} \ge 2 \sum_i \frac{F^{ii}(a_m)_i^2}{a_m} - C a_m + \sum_i f_{p_i} (a_m)_i
\end{equation}
at $x_0$ in the classical sense.
\end{proposition}

\begin{remark}
    When $(n,k)\neq (3,1)$, the constant $K_0$ can be chosen as $1$.
\end{remark}

\begin{proof}
Throughout the proof, $M(n,k)$ denotes a positive constant depending on $n$ and $k$, and $C$ denotes a positive constant depending only on $n, k,  \|u\|_{C^{0,1}(B_2(0))}, \|1/f\|_{L^\infty}$, and $\|f\|_{C^{1,1}}$. These constants may change from line to line.

Note that $b_m$ is symmetric in $\lambda_1, \dots, \lambda_m$, and $\lambda_m > \lambda_{m+1}$. Hence, the function is smooth near $x_0$.

We assume $D^2u$ is diagonalized at the point $x_0$. We will perform all computations below at $x_0$. Now, assume $\lambda \ge K_0 \ge 1$, where $K_0$ is a constant depending on $n, k, \|u\|_{C^{0,1}(B_2(0))}, \|1/f\|_{L^\infty}$, and $\|f\|_{C^{1,1}}$, which will be chosen later. We employ the approximation argument used in Lemma 2.3 of \cite{WangYu2014SLE}. We first assume that the first $m$ eigenvalues are distinct. We will compute $\sum_i F^{ii} ((b_m)_i)^2$ and $\sum_i F^{ii} (b_m)_{ii}$ by differentiating the characteristic equation $\det(D^2u - \lambda_j I) = 0$ at $x_0$.
The derivatives of the eigenvalues are given by
\begin{align*}
\partial_e \lambda_j &= u_{jje}, \\
\partial_{ee} \lambda_j &= u_{jjee} + \sum_{k \neq j} \frac{2(u_{jke})^2}{\lambda_j - \lambda_k}.
\end{align*}
Now, computing the first derivative squared term contracted with $F^{ii}$, we have
\begin{align*}
\sum_i F^{ii}((b_m)_i)^2 = \sum_i F^{ii} \left( \frac{1}{m} \sum_{a=1}^m \frac{(\lambda_a)_i}{\lambda_a} \right)^2 =  \sum_i  F^{ii} \left( \frac{1}{m} \sum_{a=1}^m \frac{u_{aai}}{\lambda_a} \right)^2.
\end{align*}
Next, we compute the second derivative term traced with $F^{ii}$ to obtain
\begin{equation}
\label{m1}
\begin{split}
 \sum_i  F^{ii}(b_m)_{ii} &=  \sum_i  F^{ii} \left( \frac{1}{m} \sum_{a=1}^m \left( \frac{(\lambda_a)_{ii}}{\lambda_a} - \frac{((\lambda_a)_i)^2}{\lambda_a^2} \right) \right) \\
&= \frac{1}{m}  \sum_i  \sum_{a=1}^m F^{ii} \left( \frac{u_{aaii}}{\lambda_a} + \sum_{k \neq a} \frac{2u_{aki}^2}{\lambda_a(\lambda_a - \lambda_k)} - \frac{u_{aai}^2}{\lambda_a^2} \right).
\end{split}
\end{equation}
Now, we expand the second term to obtain
\begin{equation*}
    \sum_{a=1}^m  \sum_{k \neq a} \frac{2u_{aki}^2}{\lambda_a(\lambda_a - \lambda_k)}  = \sum_{a=1}^m  \sum_{\alpha > m} \frac{2u_{a\alpha i}^2}{\lambda_a(\lambda_a - \lambda_\alpha)} + \sum_{\substack{a,b=1 \\ a \neq b}}^m \frac{2u_{abi}^2}{\lambda_a(\lambda_a - \lambda_b)}.
\end{equation*}
Furthermore, 
\begin{align*}
    \sum_{\substack{a,b=1 \\ a \neq b}}^m \frac{2u_{abi}^2}{\lambda_a(\lambda_a - \lambda_b)} &= \sum_{1 \le a < b \le m}   \frac{2u_{abi}^2}{\lambda_a(\lambda_a - \lambda_b)}+\sum_{1 \le a < b \le m}   \frac{2 u_{abi}^2}{\lambda_b(\lambda_b - \lambda_a)} \\
    &= -\sum_{1 \le a < b \le m} \frac{2u_{abi}^2}{\lambda_a\lambda_b} = -\sum_{\substack{a,b=1 \\ a \neq b}}^m \frac{u_{abi}^2}{\lambda_a\lambda_b}.
\end{align*}
(\ref{m1}) becomes 

\begin{equation}
\begin{split}
 &\sum_i  mF^{ii}(b_m)_{ii}=\\
 &  \sum_i F^{ii} \left( \sum_{a=1}^m \left[ \frac{u_{aaii}}{\lambda_a} + \sum_{\alpha > m} \frac{2u_{a\alpha i}^2}{\lambda_a(\lambda_a - \lambda_\alpha)} - \frac{u_{aai}^2}{\lambda_a^2} \right] -\sum_{\substack{a,b=1 \\ a \neq b}}^m \frac{u_{abi}^2}{\lambda_a\lambda_b} \right).
\end{split}
\end{equation}
Now, as in the proof of \cite{WangYu2014SLE} Lemma 2.3, since $b_m$ is $C^2$ near $D^2u(x_0)$ as a function of the matrix, we may approximate $D^2u(x_0)$ by matrices with distinct eigenvalues and take the limit (i.e., all $\lambda_a,\lambda_b \rightarrow \lambda$ if $a,b \leq m$). We have
\begin{equation}
\label{eq3_4}
\begin{split}
 &\sum_i  mF^{ii}(b_m)_{ii} =\\
 &\sum_i F^{ii} \left( \sum_{a=1}^m \left[ \frac{u_{aaii}}{\lambda} + \sum_{\alpha > m} \frac{2u_{a\alpha i}^2}{\lambda(\lambda - \lambda_\alpha)} - \frac{u_{aai}^2}{\lambda^2} \right] -\sum_{\substack{a,b=1 \\ a \neq b}}^m \frac{u_{abi}^2}{\lambda^2} \right).
\end{split}
\end{equation}
Differentiating the equation $F(D^2u) = f(x,u,Du)$ twice yields
\begin{equation*} 
\sum_i F^{ii} u_{iiaa} + \sum_{p,q,r,s} F^{pq,rs} u_{pqa} u_{rsa} = \frac{d^2}{dx_a^2}f(x,u,Du) \quad \text{for } a = 1, \dots, m.
\end{equation*}
Observe that 
\begin{equation*}
    \sum_{p,q,r,s} F^{pq,rs} u_{pqa} u_{rsa} = \sum_{p,q} F^{pp,qq}u_{ppa}u_{qqa}+ \sum_{p \neq q} F^{pq,qp}u_{pqa}^2.
\end{equation*}
(\ref{eq3_4}) becomes
\begin{equation}
\label{e3_5}
\begin{split}
\sum_i mF^{ii}(b_m)_{ii} &= \sum_{a=1}^m \sum_{p,q} -F^{pp,qq}\frac{u_{ppa}u_{qqa}}{\lambda}- \sum_{a=1}^m \sum_{p \neq q} F^{pq,qp}\frac{u_{pqa}^2}{\lambda} \\ &+ \sum_i F^{ii} \left(  \sum_{a=1}^m  \sum_{\alpha > m} \frac{2u_{a\alpha i}^2}{\lambda(\lambda - \lambda_\alpha)} -\sum_{\substack{a,b=1 \\ a \neq b}}^m \frac{u_{abi}^2}{\lambda^2} - \sum_{a=1}^m\frac{u_{aai}^2}{\lambda^2} \right) \\
&+ \sum_{a=1}^m \frac{1}{\lambda}\frac{d^2}{dx_a^2}f(x,u,Du).
\end{split}
\end{equation}

Now, we estimate the terms in \eqref{e3_5}.
For the first term, we make use of the concavity inequality in Lemma \ref{concave_simplified} to obtain
\begin{equation}
\label{eq3_6}
\begin{split}
    -\sum_{a=1}^m \sum_{p,q} F^{pp,qq}\frac{u_{ppa}u_{qqa}}{\lambda} &\ge \sum_{a=1}^m \frac{-1}{F\lambda}\left(\sum_i F^{ii}u_{iia}\right)^2+\sum_{a=1}^m (M(n,k)+1)\frac{F^{aa}u_{aaa}^2}{\lambda^2} \\
    &=\sum_{a=1}^m \frac{-1}{F\lambda}\left(\frac{df}{dx_a}\right)^2+\sum_{a=1}^m (M(n,k)+1)\frac{F^{aa}u_{aaa}^2}{\lambda^2}.
\end{split}
\end{equation}
Let $I_{1a}$ and $I_{1b}$ denote the first and second terms on the right hand side respectively.
Note that $\lambda_a = \lambda_1$ for $1 \le a \le m$ and $F^{aa} = F^{11}$ so we can apply the lemmas to all $1 \le a \le m$. This can be seen by rewriting the original proof or considering permutations of indices.

For the second term, using the fact that $-F^{pq,qp}>0$ by Lemma \ref{lem:2.2}, we drop the terms where both indices $p$ and $q$ are greater than $m$ to obtain
\begin{equation*}
\begin{split}
    -\sum_{a=1}^m \sum_{p \neq q} F^{pq,qp}\frac{u_{pqa}^2}{\lambda} &\geq -2\sum_{\substack{a,b=1 \\ a \neq b}}^m F^{ab,ba}\frac{u_{aab}^2}{\lambda}  +2\sum_{a=1}^m \sum_{\alpha > m} \frac{F^{\alpha\alpha}-F^{aa}}{\lambda-\lambda_\alpha}\frac{u_{aa\alpha}^2}{\lambda}\\
    &+ 2\sum_{\substack{a,b=1 \\ a \neq b}}^m \sum_{\alpha > m} \frac{F^{\alpha\alpha}-F^{bb}}{\lambda-\lambda_\alpha}\frac{u_{ab\alpha}^2}{\lambda} -\sum_{\substack{a,b,c=1 \\ a \neq b,a\neq c , b \neq c}}^m F^{bc,cb}\frac{u_{abc}^2}{\lambda} .
\end{split}
\end{equation*}
Let $I_{2a}$, $I_{2b}$, $I_{2c}$, and $I_{2d}$ denote the first, second, third, and fourth terms on the right hand side respectively. Note that the coefficient of $I_{2d}$ is $1$ instead of $2$.

For the third term, we drop all terms for $i>m$ to obtain
\begin{equation}
\label{third_term}
     \sum_i F^{ii}  \sum_{a=1}^m  \sum_{\alpha > m} \frac{2u_{a\alpha i}^2}{\lambda(\lambda - \lambda_\alpha)} \ge  \sum_{a=1}^m F^{aa}   \sum_{\alpha > m} \frac{2u_{aa\alpha }^2}{\lambda(\lambda - \lambda_\alpha)} + \sum_{\substack{a,b=1 \\ a \neq b}}^m \sum_{\alpha > m} \frac{2F^{bb}u_{ab\alpha }^2}{\lambda(\lambda - \lambda_\alpha)}. 
\end{equation}
Let $I_{3a}$ and $I_{3b}$ denote the first and second terms on the right hand side respectively.

For the fourth term, we expand and group to obtain
\begin{equation*}
\begin{split}
    &- \sum_i F^{ii} \sum_{\substack{a,b=1 \\ a \neq b}}^m \frac{u_{abi}^2}{\lambda^2} \\
    &= -\sum_{\substack{a,b=1 \\ a \neq b}}^m  \frac{F^{aa}u_{aab}^2+F^{bb}u_{abb}^2}{\lambda^2} -\sum_{\substack{a,b=1 \\ a \neq b}}^m \sum_{\alpha > m} F^{\alpha\alpha} \frac{u_{ab\alpha}^2}{\lambda^2} -  \sum_{\substack{a,b,c=1 \\ a \neq b,a \neq c, b \neq c}}^m F^{aa}\frac{u_{abc}^2}{\lambda^2} \\
    &= -\sum_{\substack{a,b=1 \\ a \neq b}}^m F^{aa} \frac{2u_{aab}^2}{\lambda^2} -\sum_{\substack{a,b=1 \\ a \neq b}}^m \sum_{\alpha > m} F^{\alpha\alpha} \frac{u_{ab\alpha}^2}{\lambda^2} -  \sum_{\substack{a,b,c=1 \\ a \neq b,a \neq c, b \neq c}}^m F^{aa}\frac{u_{abc}^2}{\lambda^2} .
\end{split}
\end{equation*}
Let $I_{4a}$, $I_{4b},$ and $I_{4c}$ denote the first, second, and third terms on the last line respectively.

For the fifth term, we similarly expand to obtain
\begin{equation*}
    -\sum_{a=1}^m   \sum_i F^{ii}  \frac{u_{aai}^2}{\lambda^2} = -\sum_{a=1}^m F^{aa}  \frac{u_{aaa}^2}{\lambda^2} -\sum_{\substack{a,b=1 \\ a \neq b}}^m F^{bb}  \frac{u_{aab}^2}{\lambda^2} -\sum_{a=1}^m   \sum_{\alpha > m} F^{\alpha\alpha}  \frac{u_{aa\alpha}^2}{\lambda^2}.  
\end{equation*}
Let $I_{5a}$, $I_{5b}$, and $I_{5c}$ denote the first, second, and third terms on the right hand side respectively.

Now, we group the terms by the indices of the third derivatives. 

First, we group the terms involving $u_{aaa}^2$.
Combining the corresponding terms in $I_{1b}$ and $I_{5a}$, we have 
\begin{equation}
\label{u_aaa}
\begin{split}
    &\sum_{a=1}^m (M(n,k)+1)\frac{F^{aa}u_{aaa}^2}{\lambda^2}-  \sum_{a=1}^m \frac{F^{aa}u_{aaa}^2}{\lambda^2} \\
    &=   \sum_{a=1}^m M(n,k) \frac{F^{aa}u_{aaa}^2}{\lambda^2}.
\end{split}
\end{equation}

Now, we group the terms involving $u_{aab}^2$ or $u_{abb}^2$.
Combining $I_{2a}$, $I_{4a}$, and $I_{5b}$, we have 
\begin{equation}
\label{u_aab_prep}
\begin{split}
   -2\sum_{\substack{a,b=1 \\ a \neq b}}^m F^{ab,ba}\frac{u_{aab}^2}{\lambda} -3\sum_{\substack{a,b=1 \\ a \neq b}}^m F^{bb}  \frac{u_{aab}^2}{\lambda^2}.
\end{split}
\end{equation}

Now, we claim that under our assumption $\lambda \ge K_0$, where $K_0 \ge 1$ depends only on  $n, k, \|u\|_{C^{0,1}(B_2(0))}, \|1/f\|_{L^\infty}$, and $\|f\|_{C^{1,1}}$, the term (\ref{u_aab_prep}) has the following lower bound
\begin{equation}
\label{u_aab}
M(n,k)\sum_{\substack{a,b=1 \\ a \neq b}}^m F^{bb}  \frac{u_{aab}^2}{\lambda^2}.
\end{equation}

Case $k=n-1$.
By Lemma \ref{lem:F_abba_exact}(a), $ -F^{ab,ba} =  \frac{2F^{aa}}{\lambda}$. (\ref{u_aab_prep}) becomes 
\begin{equation*}
\sum_{\substack{a,b=1 \\ a \neq b}}^m F^{bb}  \frac{u_{aab}^2}{\lambda^2}.
\end{equation*}

Case $k=n-2$. By Lemma \ref{lem:F_abba_exact}(b), $ -F^{ab,ba} = \frac{2F^{aa}}{\lambda} - \frac{F^2}{\lambda^4}$. (\ref{u_aab_prep}) becomes
\begin{equation}
\label{u_aab_n-2_prep2}
    \begin{split}
        \sum_{\substack{a,b=1 \\ a \neq b}}^m \left(\frac{F^{aa}}{\lambda}  - 2\frac{F^2}{\lambda^4}\right)\frac{u_{aab}^2}{\lambda} =  \sum_{\substack{a,b=1 \\ a \neq b}}^m \left(\frac{F^2}{\lambda^3}\sum_{i \neq a} \frac{1}{\lambda_i} -2\frac{F^2}{\lambda^4}\right)\frac{u_{aab}^2}{\lambda}
    \end{split}
\end{equation}

Subcase $n = 3$ and $m=1$.
The term does not exist.

Subcase $n=3$ and $m=2$.
By Lemma \ref{lu_eigenvalue_bounds}, $\lambda_3\le \sqrt{C(3)f} \le C \sqrt{\sup f} =: c_0$. Suppose that $\lambda \geq 2c_0$, then 
\begin{equation*}
    \frac{F^2}{\lambda^3}\sum_{i \neq a} \frac{1}{\lambda_i} -2\frac{F^2}{\lambda^4} =\frac{F^2}{\lambda^3}\frac{1}{\lambda_3} -\frac{F^2}{\lambda^4}  \geq \frac{F^2}{\lambda^3}\frac{1}{2 \lambda_3} \geq  \frac{F^2}{\lambda^3}\sum_{i \neq a} \frac{1}{4\lambda_i} = \frac{F^{aa}}{4\lambda} \ge M(3,1)\frac{F^{aa}}{\lambda}.
\end{equation*}
We have the desired lower bound.

Subcase $n=3$ and $m=3$.
Indeed, we do not need the Jacobi inequality. From the equation, it is immediate that $\lambda = \sqrt{3f} \le  \sqrt{3\sup f} $.

Subcase $n\ge 4$.
\begin{equation*}
    \frac{F^2}{\lambda^3}\sum_{i \neq a} \frac{1}{\lambda_i} -2\frac{F^2}{\lambda^4} \ge \frac{F^2}{\lambda^3} \frac{1}{\lambda_n} \ge  \frac{F^2}{\lambda^3} \frac{1}{n-1} \sum_{i \neq a} \frac{1}{\lambda_i} = \frac{1}{(n-1)\lambda} F^{aa} \ge M(n,n-2)\frac{ F^{aa} }{\lambda}.
\end{equation*}
We have the desired lower bound.

Hence, \eqref{u_aab_prep} is bounded from below by \eqref{u_aab} if $\lambda \ge K_0$.

Next, we group the terms involving $u_{abc}^2$. By combining $I_{2d}$ and $I_{4c}$, we arrive at the expression
\begin{align}
\label{u_abc}
    -\sum_{\substack{a,b,c=1 \\ a \neq b,a\neq c , b \neq c}}^m F^{bc,cb}\frac{u_{abc}^2}{\lambda} -  \sum_{\substack{a,b,c=1 \\ a \neq b,a \neq c, b \neq c}}^m F^{aa}\frac{u_{abc}^2}{\lambda^2}.
\end{align}
Now, we claim that the expression in \eqref{u_abc} is non-negative.

Case $k= n-1$. By Lemma \ref{lem:F_abba_exact}(a), $ -F^{ab,ba} =  \frac{2F^{aa}}{\lambda}$. (\ref{u_abc}) becomes 
\begin{equation*}
\sum_{\substack{a,b,c=1 \\ a \neq b,a \neq c, b \neq c}}^m F^{aa}\frac{u_{abc}^2}{\lambda^2} \geq 0.
\end{equation*}

Case $k = n-2$. By Lemma \ref{lem:F_abba_exact}(b), $ -F^{ab,ba} = \frac{2F^{aa}}{\lambda} - \frac{F^2}{\lambda^4}$. 
(\ref{u_abc}) becomes 
\begin{align*}
  &\sum_{\substack{a,b,c=1 \\ a \neq b,a \neq c, b \neq c}}^m \left(\frac{F^{aa}}{\lambda^2} -\frac{F^2}{\lambda^5}\right)  u_{abc}^2 \\
  &= \sum_{\substack{a,b,c=1 \\ a \neq b,a \neq c, b \neq c}}^m \left(\frac{F^2}{\lambda^4}\sum_{i \neq a} \frac{1}{\lambda_i} -\frac{F^2}{\lambda^5}\right)  u_{abc}^2 \ge 0,
\end{align*}
where we use Lemma \ref{n-2properties} in the first equality.

Next, we group the terms involving $u_{ab\alpha}^2$.
Combining $I_{2c}$, $I_{3b}$, and $I_{4b}$, we have
\begin{equation}
\label{u_abalpha}
\begin{split}
    &2 \sum_{\substack{a,b=1 \\ a \neq b}}^m \sum_{\alpha > m} \frac{F^{\alpha\alpha}-F^{bb}}{\lambda-\lambda_\alpha}\frac{u_{ab\alpha}^2}{\lambda}+\sum_{\substack{a,b=1 \\ a \neq b}}^m \sum_{\alpha > m} \frac{2F^{bb}u_{ab\alpha }^2}{\lambda(\lambda - \lambda_\alpha)}   -\sum_{\substack{a,b=1 \\ a \neq b}}^m \sum_{\alpha > m} F^{\alpha\alpha} \frac{u_{ab\alpha}^2}{\lambda^2}\\
    &= \sum_{\substack{a,b=1 \\ a \neq b}}^m \sum_{\alpha > m} \frac{2F^{\alpha\alpha}}{\lambda-\lambda_\alpha}\frac{u_{ab\alpha}^2}{\lambda}   -\sum_{\substack{a,b=1 \\ a \neq b}}^m \sum_{\alpha > m} F^{\alpha\alpha} \frac{u_{ab\alpha}^2}{\lambda^2}\\
    &\ge \sum_{\substack{a,b=1 \\ a \neq b}}^m \sum_{\alpha > m} F^{\alpha\alpha} \frac{u_{ab\alpha}^2}{\lambda^2} \geq 0.
\end{split}
\end{equation}

Finally, we group the terms involving $u_{aa\alpha}^2$.
Combining $I_{2b}$, $I_{3a}$, and $I_{5c}$, we have
\begin{equation}
\label{u_aaalpha}
    \begin{split}
    &2\sum_{a=1}^m \sum_{\alpha > m} \frac{F^{\alpha\alpha}-F^{aa}}{\lambda-\lambda_\alpha}\frac{u_{aa\alpha}^2}{\lambda}+ \sum_{a=1}^m F^{aa}   \sum_{\alpha > m} \frac{2u_{aa\alpha }^2}{\lambda(\lambda - \lambda_\alpha)}-\sum_{a=1}^m   \sum_{\alpha > m} F^{\alpha\alpha}  \frac{u_{aa\alpha}^2}{\lambda^2}  \\
    &= 2\sum_{a=1}^m \sum_{\alpha > m} \frac{F^{\alpha\alpha}}{\lambda-\lambda_\alpha}\frac{u_{aa\alpha}^2}{\lambda}-\sum_{a=1}^m   \sum_{\alpha > m} F^{\alpha\alpha}  \frac{u_{aa\alpha}^2}{\lambda^2}  \\
    &\ge\sum_{a=1}^m   \sum_{\alpha > m} F^{\alpha\alpha}  \frac{u_{aa\alpha}^2}{\lambda^2}.  
    \end{split}
\end{equation}

Since we showed that the terms in (\ref{u_abc}) and (\ref{u_abalpha}) are non-negative, they can be dropped. 
Substituting the lower bounds (\ref{u_aaa}), (\ref{u_aab}) and (\ref{u_aaalpha}) back into (\ref{e3_5}), we obtain
\begin{equation}
\label{last3step}
\begin{split}
\sum_i mF^{ii}(b_m)_{ii} &\ge M(n,k) \sum_{a=1}^m \sum_i F^{ii} \frac{u_{aai}^2}{\lambda^2}  + \sum_{a=1}^m \frac{1}{\lambda}\frac{d^2}{dx_a^2}f(x,u,Du) \\
&-\sum_{a=1}^m \frac{1}{F\lambda}\left(\frac{df}{dx_a}\right)^2. 
\end{split} 
\end{equation}

By the Cauchy-Schwarz inequality, 
\begin{equation}
\label{cauchyschwarz}
    \begin{split}
    \sum_i F^{ii} \sum_{a=1}^m \frac{u_{aai}^2}{\lambda^2} &= \sum_{a=1}^m \sum_i F^{ii} (\ln \lambda_a)_i^2\\
    &\ge \frac{1}{m} \sum_i F^{ii} \left(\sum_{a=1}^m \ln \lambda_a\right)_i^2 = m \sum_i F^{ii} (b_m)_i^2.        
    \end{split}
\end{equation}

Expanding the total derivatives $\frac{df}{dx_a}$ and $\frac{d^2f}{dx_a^2}$ in (\ref{last3step}), and applying the Cauchy-Schwarz bound (\ref{cauchyschwarz}), we obtain 
\begin{equation}
\label{jacobi_final}
\begin{split}
\sum_i mF^{ii}(b_m)_{ii} &\ge M(n,k) m\sum_i F^{ii} (b_m)_i^2  + \sum_{a=1}^m \frac{\partial^2 f}{\partial p_a^2} u_{aa} + \sum_{a=1}^m \sum_i \frac{\partial f}{\partial p_i} (\ln \lambda_a)_i \\
&-\sum_{a=1}^m \frac{1}{F}\left(\frac{\partial f}{\partial p_a}\right)^2 u_{aa} - C - \frac{C}{\lambda}.
\end{split} 
\end{equation}
Using the assumptions that $\lambda \ge K_0 \ge 1$ and that $\log f $ is convex in $Du$, and dividing both sides by $m$, we have
\begin{equation*}
\begin{split}
\sum_i F^{ii}(b_m)_{ii} &\ge M(n,k) \sum_i F^{ii} (b_m)_i^2   +  \sum_i \frac{\partial f}{\partial p_i} (b_m)_i -C. \\
\end{split} 
\end{equation*}
\end{proof}

The following proposition is a modification of \cite{LuTsai_general_quotient} Lemma 2.5.
\begin{proposition}
\label{jacobi_general_l_k}
Let $n \ge 2$, $1 \le k < l \le n$ with $k \in \{l-1, l-2\}$, and $f > 0$ be a $C^{1,1}$ function satisfying the structural inequality $\sum_{i,j} f_{p_i p_j} \xi_i \xi_j \ge \frac{K_1}{f} \left(\sum_i f_{p_i} \xi_i\right)^2$ for all $\xi \in \mathbb{R}^n$. Suppose Assumption \ref{assumption1.1} holds, $u \in C^4(B_{2}(0))$ is a  strictly convex solution to \eqref{main_equation_2}, and at a fixed point $x_0 \in B_2(0)$, $D^2u$ is diagonalized with its eigenvalues ordered as 
\begin{equation*}
    \lambda = \lambda_1 = \dots = \lambda_m > \lambda_{m+1} \ge \dots \ge \lambda_n \ge 0
\end{equation*}
Then, the function $b_m := \frac{1}{m} \sum_{a=1}^m \ln \lambda_a$ is locally smooth near $x_0$. Furthermore, there exists a constant $K_0 \ge 1$, depending only on $n, l, k, \|u\|_{C^{0,1}(B_2(0))}, \\ \|1/f\|_{L^\infty}$, and $\|f\|_{C^{1,1}}$, such that if $\lambda \ge K_0$, $b_m$ satisfies the Jacobi inequality 
\begin{equation*}
\sum_i F^{ii}(b_m)_{ii} \ge \tilde{c}(n,l)\sum_i F^{ii}((b_m)_i)^2 + \sum_i f_{p_i} (b_m)_i  - C
\end{equation*}
at $x_0$ in the classical sense. Here, $\tilde{c}(n,l):= \min(\frac{1}{2},c(n,l))$, where $c(n,l)$ is the constant from Assumption \ref{assumption1.1}, and $C$ is a constant depending only on $n, l, \|u\|_{C^{0,1}(B_2(0))}, \|1/f\|_{L^\infty}$, and $\|f\|_{C^{1,1}}$.

Equivalently, under the same assumptions, the function  $a_m:= \exp(\tilde{c}(n,l) b_m)$ satisfies the Jacobi inequality
\begin{equation}
\label{a_jacobi_2}
\sum_i F^{ii} (a_m)_{ii} \ge 2 \sum_i \frac{F^{ii}(a_m)_i^2}{a_m} - C a_m + \sum_i f_{p_i} (a_m)_i
\end{equation}
at $x_0$ in the classical sense.
\end{proposition}

\begin{proof}
Throughout the proof, $C$ and $K_0$ denote positive constants that may change from line to line, depending only on $n, l, \|u\|_{C^{0,1}(B_2(0))}, \|1/f\|_{L^\infty}$, and $\|f\|_{C^{1,1}}$.  Now, we assume $\lambda \ge K_0 \ge 1$, where $K_0$ is a constant depending on $n, l,  \|u\|_{C^{0,1}(B_2(0))},\|1/f\|_{L^\infty}$, and $\|f\|_{C^{1,1}}$ to be chosen later.
    The proof is highly similar to the proof of Proposition \ref{jacobi}, except for the following three modifications.
    
    First, in (\ref{eq3_6}), instead of using the concavity inequalities Lemma \ref{concave_simplified}, we use Assumption \ref{assumption1.1} to obtain
\begin{equation*}
\begin{split}
    -\sum_{a=1}^m \sum_{p,q} F^{pp,qq}\frac{u_{ppa}u_{qqa}}{\lambda} 
    &\ge -\sum_{a=1}^m \frac{K_1}{F\lambda}\left(\frac{df}{dx_a}\right)^2+\sum_{a=1}^m (c(n,l)+1)\frac{F^{aa}u_{aaa}^2}{\lambda^2} \\
    &-\sum_{a=1}^m\sum_{\alpha>m}\frac{2F^{\alpha\alpha}u_{a\alpha \alpha}^2}{\lambda(\lambda-\lambda_\alpha)}.
\end{split}
\end{equation*}
Note that $\lambda_a = \lambda_1$ for $1 \le a \le m$ and $F^{aa} = F^{11}$, so we can apply Assumption \ref{assumption1.1} to all $1 \le a \le m$. This can be seen by considering permutations of indices.
Note that there will be an extra term $-\sum_{a=1}^m\sum_{\alpha>m}\frac{2F^{\alpha\alpha}u_{a\alpha\alpha}^2}{\lambda(\lambda-\lambda_\alpha)}$, which can be absorbed by the $i>m$ term in (\ref{third_term}) that we dropped in the proof of Proposition \ref{jacobi}.
    
Second, as in Proposition \ref{jacobi}, we need to show that the terms involving $u_{aab}^2$ or $u_{abb}^2$ have the following lower bound
\begin{equation}
\label{u_aab_prep_l}
\begin{split}
   -2\sum_{\substack{a,b=1 \\ a \neq b}}^m F^{ab,ba}\frac{u_{aab}^2}{\lambda} -3\sum_{\substack{a,b=1 \\ a \neq b}}^m F^{bb}  \frac{u_{aab}^2}{\lambda^2} \geq \frac{1}{2}\sum_{\substack{a,b=1 \\ a \neq b}}^m F^{bb}  \frac{u_{aab}^2}{\lambda^2},
\end{split}
\end{equation}
where $c(n,l)$ is the constant from Assumption \ref{assumption1.1}. Note that the constant was $M(n,k)$ instead of $\frac{1}{2}$ in the proof of Proposition \ref{jacobi}.  
We first notice that by Lemma \ref{lem:2.3} and Lemma \ref{lem:2.4}, $\lambda_l \leq C$. Hence, we may assume $m\le l-1$ else $\lambda$ is trivially bounded. We may also assume $m > 1$, else the sums on both sides of (\ref{u_aab_prep_l}) are equal to $0$. 
Now, we fix $a,b$. Subtracting the right hand side of (\ref{u_aab_prep_l}) from the left, it suffices to show that the resulting coefficient of $\frac{u_{aab}^2}{\lambda^2}$ is non-negative. This coefficient is given by
\begin{equation}
\label{u_aab_l-1_1}
    -2\lambda F^{ab,ba} - \frac{7}{2}F^{bb} = 2(F^{bb}-G^{ab})- \frac{7}{2}F^{bb}= -\frac{3}{2}F^{bb}-2G^{ab},
\end{equation}
where we use Lemma \ref{lem:F_abba_general_l_k} in the first equality and $G^{ab}$ is defined in (\ref{G_def}).

 For simplicity, we use the following notation in the proof.
\begin{equation}
\label{T_j}
    T_j := \sigma_j(\lambda_1,\lambda_2,\cdots, \widehat{\lambda_a}, \cdots  \widehat{\lambda_b}, \cdots \lambda_n) 
\end{equation}
As mentioned before, $C \ge \lambda_l \ge \lambda_{l+1} \ge  \cdots \ge \lambda_n$. Note that $\lambda_{m+1},\cdots,\lambda_{l-1}$ may not be bounded. We will have to keep track of them.

Applying the equality $\sigma_{j}(\lambda) = \lambda_i \sigma_{j-1}(\lambda| \lambda_i) + \sigma_{j}(\lambda|\lambda_i)$ to (\ref{G_def}) and (\ref{Faa_l}), we have
\begin{align}
    -G^{ab} \sigma_{k}^2 &= \lambda^2(T_{l-2}T_{k-1} - T_{k-2} T_{l-1}) + T_{l}T_{k-1}-T_{l-1} T_{k}, \\
    F^{aa} \sigma_{k}^2  &= \lambda^2(T_{l-2}T_{k-1} - T_{k-2} T_{l-1}) + \lambda(T_{k}T_{l-2}-T_l T_{k-2}) + T_{l-1}T_k-T_l T_{k-1}.
\end{align}
(\ref{u_aab_l-1_1}) becomes
\begin{equation}
\label{u_aab_l-1_2}
    \begin{split}
        (-\frac{3}{2}F^{bb}-2G^{ab})\sigma_{k}^2 &= \frac{1}{2}\lambda^2(T_{l-2}T_{k-1} - T_{k-2} T_{l-1}) - \frac{3}{2}\lambda(T_{k}T_{l-2}-T_l T_{k-2})\\
        &- \frac{7}{2}(T_{l-1}T_k-T_l T_{k-1}) \\
        &\ge C\lambda^2 T_{l-2}T_{k-1} - \frac{3}{2}\lambda T_{k}T_{l-2} - \frac{7}{2}T_{l-1}T_k,
    \end{split}
\end{equation}
where we use Newton-Maclaurin inequalities on the first term and we dropped the positive terms in the second and third terms in the third inequality. 

Now, we split into the case $k=l-1$ and the case $k= l-2$.

Case $k = l-1$. We drop all the terms except the largest term.
\begin{equation*}
\begin{split}
    T_{l-2}T_{k-1} &\ge \lambda^{2m-4}\lambda_{m+1}^2 \ldots \lambda_{l-1} ^2 \lambda_l^2  \\
    &\ge C \lambda^{2m-4}\lambda_{m+1}^2  \ldots \lambda_{l-1} ^2,
\end{split}
\end{equation*}
where we use $\lambda_l \ge C$.
For the negative terms, they are bounded below by a negative multiple of the largest term.
\begin{equation*}
     -T_{k}T_{l-2} \ge -C \lambda^{2m-4} \lambda_{m+1}^2 \ldots \lambda_{l}^2 \lambda_{l+1} \ge  -C \lambda^{2m-4} \lambda_{m+1}^2 \ldots \lambda_{l-1}^2,
\end{equation*}
where we use $\lambda_{l+1}\le \lambda_{l} \le C$.
Similarly,
\begin{equation*}
     -T_{l-1}T_{k} \ge  -C \lambda^{2m-4} \lambda_{m+1}^2 \ldots \lambda_{l-1}^2.
\end{equation*}
(\ref{u_aab_l-1_2}) becomes
\begin{equation}
\label{u_aab_l-1_3n-1}
     (-\frac{3}{2}F^{bb}-2G^{ab})\sigma_{l-1}^2 \ge C \lambda_{m+1}^2\ldots \lambda_{l-1}^2 (\lambda^{2m-2}-C\lambda^{2m-3}-C\lambda^{2m-4}) >0,
\end{equation}
if $\lambda>K_0$.

Case $k = l-2$. We drop all the terms except the largest term.
\begin{equation*}
\begin{split}
    T_{l-2}T_{k-1} &\ge \lambda^{2m-4}\lambda_{m+1}^2  \ldots  \lambda_{l-1} ^2  \lambda_{l}.
\end{split}
\end{equation*}
For the negative terms, they are bounded below by a negative multiple of the largest term
\begin{equation*}
     -T_{k}T_{l-2} \ge -C \lambda^{2m-4} \lambda_{m+1}^2 \ldots \lambda_{l}^2 \ge  -C \lambda^{2m-4} \lambda_{m+1}^2 \ldots \lambda_{l-1}^2 \lambda_l,
\end{equation*}
where we use $\lambda_l \le C$.
Similarly,
\begin{equation*}
     -T_{l-1}T_{k} \ge  -C \lambda^{2m-4} \lambda_{m+1}^2 \ldots \lambda_{l-1}^2 \lambda_{l}.
\end{equation*}
(\ref{u_aab_l-1_2}) becomes
\begin{equation}
\label{u_aab_l-1_3n-2}
     (-\frac{3}{2}F^{bb}-2G^{ab})\sigma_{k}^2 \ge C \lambda_{m+1}^2\ldots \lambda_{l-1}^2 \lambda_l (\lambda^{2m-2}-C\lambda^{2m-3}-C\lambda^{2m-4}) >0,
\end{equation}
if $\lambda>K_0$.

Third, as before, the terms involving $u_{abc}^2$ are given by (\ref{u_abc}). We need to show 
\begin{equation*}
       -F^{bc,cb}\lambda -F^{aa} \ge 0.
\end{equation*}

Using Lemma \ref{lem:F_abba_general_l_k} and the definition of $G^{ab}$ in (\ref{G_def}), we have 
\begin{equation*}
    -F^{bc,cb}\lambda -F^{aa} = F^{bb}-G^{bc}-F^{aa} = -G^{bc} = -G^{ab} \ge 0,
\end{equation*}
where in the last inequality we use the fact that $-G^{ab}\ge \frac{3}{4}F^{bb}\ge 0$ established in (\ref{u_aab_l-1_3n-1}) and (\ref{u_aab_l-1_3n-2}).

Finally, instead of (\ref{jacobi_final}), we have the following inequality
\begin{equation*}
\begin{split}
\sum_i mF^{ii}(b_m)_{ii} &\ge \min(\frac{1}{2},c(n,l)) m\sum_i F^{ii} (b_m)_i^2  + \sum_{a=1}^m \frac{\partial^2 f}{\partial p_a^2} u_{aa} \\ 
&+ \sum_{a=1}^m \sum_i \frac{\partial f}{\partial p_i} (\ln \lambda_a)_i 
-\sum_{a=1}^m \frac{K_1}{F}\left(\frac{\partial f}{\partial p_a}\right)^2 u_{aa} - C - \frac{C}{\lambda},
\end{split} 
\end{equation*}
where $c(n,l)$ is the constant from Assumption \ref{assumption1.1}. 
Using the assumed structural inequality 
$\sum_{i,j} f_{p_i p_j} \xi_i \xi_j \ge \frac{K_1}{f} \left(\sum_i f_{p_i} \xi_i\right)^2$ for any $\xi \in \mathbb{R}^n$, the assumption $\lambda \ge 1 $, and dividing both sides by $m$, we have the desired inequality.

\end{proof}

\section{Doubling Inequality}
\label{s4}
In this section, we establish the doubling inequalities. We first derive the maximum principle inequality \eqref{max_prin_inequality} using the test function constructed by Shankar \cite{Shankar2024SLE}. As noted by Shankar, this construction is motivated by Korevaar exponential cutoff \cite{Korevaar1987} and Guan-Qiu \cite{GuanQiu2019} type radial derivative $(x-y)\cdot Du - u$. To deal with the third order term $\sum_i f_{p_i}(a_m)_i$, we use the fact that the first derivative of the test function is zero at the maximum point. A crucial observation is that this term later generates the term $h \sum_i (x_i-y_i) u_{ii}f_{p_i}$, which cancels with an identical term arising from the expansion of $ \sum_i F^{ii}\varphi_{ii}$ upon applying the identity $ \sum_i F^{ii}u_{iij} = \frac{d}{dx_j}f$. 

To obtain the final doubling inequality from \eqref{max_prin_inequality}, we proceed by cases.
The subcases $k=n-1$ and $k=l-1$ are relatively straightforward. For $k=n-1$, we simply utilize the bound $\sum_i F^{ii} \le C$, while for $k=l-1$, we apply the estimates established by Lu and Tsai \cite{LuTsai_general_quotient} to show $\sum_i F^{ii} \le C$. Then the estimates follow from a standard computation.

The subcases $k=n-2$ and $k=l-2$ are more involved because $\sum_i F^{ii}$ may not be bounded.  Instead, our strategy is to show that, for each $i$, the coefficient of $(x_i-y_i)^2$ has a lower bound $ C \sum_k F^{kk}$. To prove this for $k=n-2$, we analyze two further subcases by comparing $\frac{1}{\lambda_i}$ with $\sum_j \frac{1}{\lambda_j}$. For $k=l-2$, we make use of the fact, established by Lu and Tsai \cite{LuTsai_general_quotient}, that $F^{ii}(1+u_{ii})^2$ is bounded above and below by a constant multiple of $\lambda_{l-1}$. Finally, we use the fact that the trace $\sum_i F^{ii}$ is bounded from below and divide both sides of the maximum principle inequality by the trace to conclude the proof.

\begin{proposition}[Doubling inequality]
\label{doubling}
Let $n \ge 2$, $k \in \{n-1, n-2\}$ with $k \ge 1$, and let $f(x,u,Du)$ be a positive $C^{1,1}$ function such that $\log f$ is convex in the $Du$ variable. Suppose $u \in C^4(B_{2}(0))$ is a convex solution to \eqref{main_equation_1} on $B_2(0)$. Then, for any point $y \in B_1(0)$, there exists a sufficiently small constant $R(n,\|u\|_{C^{0,1}(B_1(0))}) > 0$ such that for any $R > r > 0$, we have
\begin{equation}
\begin{split}
    \sup_{B_R(y)} \lambda_1(D^2u)  &\leq C \left(r,n, k, \|u\|_{C^{0,1}(B_2(0))}, \left\|\frac{1}{f}\right\|_{L^\infty},\|f\|_{C^{1,1}},\sup_{B_r(y)} \lambda_1(D^2u)\right).  
\end{split}
\end{equation}
\end{proposition}

\begin{proof}
In the proof, $C$ denotes a constant that might change from line to line, depending on $r, n, k, \|u\|_{C^{0,1}(B_2(0))}, \left\|\frac{1}{f}\right\|_{L^\infty}$, and $\|f\|_{C^{1,1}}$.

We consider the cutoff function introduced in Proposition 4.1 of \cite{Shankar2024SLE}.
\begin{equation}
\label{cutoff}
   \eta(x) = \left(e^{\frac{1-\varphi(x)}{h}} - 1\right)_+, 
\end{equation}
where $\varphi(x) = (x-y) \cdot Du(x) - u(x) + u(y) + \frac{t|x-y|^2}{2}$. Here, $t$ is a positive constant to be chosen depending on $n$ and $\|u\|_{C^{0,1}(B_1(0))}$, and $h$ is a positive constant to be chosen depending on $r, n, k, \|u\|_{C^{0,1}(B_2(0))}, \left\|\frac{1}{f}\right\|_{L^\infty}$, and $\|f\|_{C^{1,1}}$.

Observe that $1-\varphi(x) \le C\|u\|_{C^{0,1}(B_1(0))} - \frac{t}{2}|x-y|^2$. There exists a sufficiently large constant $t \ge C(n,\|u\|_{C^{0,1}(B_1(0))})$ such that $1- \varphi(x) < 0$ for any $|x-y| \geq \frac{1}{2}$. We assume $t\ge 1$.

Furthermore, as $x \rightarrow y$, $\varphi(x) \rightarrow 0$. Hence, there exists a sufficiently small constant $R = C(n,t,\|u\|_{C^{0,1}(B_1(0))}) = C(n,\|u\|_{C^{0,1}(B_1(0))})$ such that if $|x-y|<R$, then $\eta(x) > 0$.
Moreover, $\inf_{x \in B_R(y)} \eta(x) \geq C$ and $\sup_{x \in B_R(y)} \eta(x) \le C$. Now, we fix $R > 0$ and let $r \le R$.

Consider the maximum point $x_0$ of $a_1(x)\eta(x)$ in $\overline{B_\frac{1}{2}(y)} \setminus B_r(y)$. There are three cases.  Case 1. The maximum point occurs in the interior and $\lambda_1(x_0) \leq K_0$, where $K_0$ is the constant in Proposition \ref{jacobi}. Case 2. The maximum point occurs in the interior and $\lambda_1(x_0) > K_0$. Case 3. The maximum point occurs on the boundary. We will show that if $h$ is sufficiently small, then case 2 is not possible. Case 1 and case 3 each lead to the desired doubling inequality.

Case 1. $\lambda_1(x_0) \leq K_0$. The Jacobi inequality in Proposition \ref{jacobi} does not apply in this case. However, $|D^2u(x_0)|$ is already bounded. Note that we have an upper bound and a lower bound for $\eta$. We obtain the desired bound.
    \begin{equation}
    \label{doublinga}
        \begin{split}
        \sup_{B_{R}(y)}a_1(x) & \leq \sup_{B_{R}(y)\backslash B_r(y)}a_1(x) + \sup_{B_{r}(y)}a_1(x) \\
        &\leq C \sup_{B_{R}(y)\backslash B_r(y)}\eta(x)a_1(x) + \sup_{B_{r}(y)}a_1(x) \\
        &\leq C \sup_{B_{R}(y)\backslash B_r(y)}\eta(x)K_0 +  \sup_{B_r(y)} a_1(x) \\
        &\leq C+ \sup_{B_r(y)} a_1(x), 
        \end{split}
    \end{equation}
Notice that the constant $C$ depends on the constant $h$, but $h$ will be fixed in case 2 and the choice of $h$ depends only on $r, n, k, \|u\|_{C^{0,1}(B_2(0))}, \left\|\frac{1}{f}\right\|_{L^\infty}$, and $\|f\|_{C^{1,1}}$.

Case 2. $\lambda_1(x_0) \ge K_0$. The Jacobi inequality holds. Suppose the multiplicity of the largest eigenvalue of $D^2u$ at $x_0$ is $m$. We first claim that $\eta(x_0) > 0$. If $\eta(x_0) = 0$, then $a_1(x_0)\eta(x_0)= 0$, which is impossible because $\eta(x) >0$ for $|x-y| \le R$ and we always have $a_1(x)>0$.

Note that $a_1(x)$ may not be smooth near $x_0$, but we have
\begin{equation*}
    a_m(x_0) \eta(x_0) = a_1(x_0)\eta(x_0) \geq a_1(x)\eta(x) \geq a_m(x) \eta(x)
\end{equation*}
in a neighborhood of $x_0$, with equality at $x_0$. Hence, the function $a_m(x) \eta(x)$ attains its local maximum at $x_0$. By Proposition \ref{jacobi}, $a_m$ is smooth near $x_0$. Also, $\eta$ is smooth near $x_0$ since $\eta(x_0) >0$.

From now on, unless otherwise specified, all computations are evaluated at $x_0$, and we assume $D^2u(x_0)$ is diagonalized.

Since $\nabla (a_m \eta) = 0$ at the maximum point, we have
\begin{equation}
\label{maxprin_1st_der}
 \eta_i = -\frac{\eta}{a_m}(a_m)_i.  
\end{equation}

Contracting $F^{ii}$ with $(a_m\eta)_{ii}$ yields
\begin{align*}
0 &\geq \sum_i F^{ii}(a_m \eta)_{ii} \\
  &= \sum_i \left(F^{ii}(a_m)_{ii} \eta + 2 F^{ii}(a_m)_i \eta_i + F^{ii} a_m \eta_{ii} \right)\\
   &= \sum_i \left( F^{ii}(a_m)_{ii} \eta - 2 F^{ii}\frac{\eta}{a_m}(a_m)_i^2 + F^{ii} a_m \eta_{ii} \right) \\
     &= \sum_i F^{ii} \left( (a_m)_{ii} - 2\frac{(a_m)_i^2}{a_m} \right)\eta +\sum_i  F^{ii} a_m \eta_{ii} \\
     &\ge  \sum_i f_{p_i} (a_m)_i \eta - C a_m \eta + \sum_i  F^{ii}a_m \eta_{ii}\\
     &= -\sum_i f_{p_i} a_m\eta_i - C a_m \eta + \sum_i  F^{ii}a_m \eta_{ii},
\end{align*}
where we substitute (\ref{maxprin_1st_der}) in the second equality, use the Jacobi inequality in Proposition \ref{jacobi} in the penultimate step, and use (\ref{maxprin_1st_der}) again in the last step. 
Because $a_m > 0$, we have
\begin{equation}
\label{eq_4_5}
    0 \ge -\sum_i f_{p_i} \eta_i - C (\eta+1) + \sum_i F^{ii}\eta_{ii}.
\end{equation}
By the definition of $\eta$, we have
\begin{equation}
\label{eq_4_6}
    \eta_i = -(\eta+1) \frac{\varphi_i}{h},
\end{equation}
and
\begin{equation}
    \label{eq_4_7}
    \eta_{ii} = -(\eta+1) \frac{\varphi_{ii}}{h}  + (\eta+1) \frac{(\varphi_{i})^2}{h^2}.
\end{equation}

Substituting (\ref{eq_4_6}) and (\ref{eq_4_7}) into (\ref{eq_4_5}) and multiplying both sides by $\frac{h^2}{\eta+1}$, we have
\begin{equation}
\label{eq_4_8}
    Ch^2+h (\sum_i F^{ii} \varphi_{ii}) \ge \sum_i F^{ii} \varphi_i^2+ h  \sum_i f_{p_i}\varphi_i.
\end{equation} 

Differentiating $\varphi(x)$, we obtain
\begin{align*}
\varphi_i &= u_i + \sum_j (x_j-y_j) u_{ij} - u_i + t (x_i-y_i) = \sum_j (x_j-y_j) u_{ij} + t (x_i-y_i).
\end{align*}
Because $D^2u$ is diagonalized at $x_0$, this simplifies to
\begin{equation}
\label{phi_d}
    \varphi_i = (x_i-y_i) (u_{ii} + t).
\end{equation}
Hence, we have
\begin{equation}
\label{phi_d_square}
    \varphi_i^2 = (x_i-y_i)^2 (u_{ii} + t)^2.
\end{equation}

The second derivative is
\begin{equation*} 
\varphi_{ii} = u_{ii} + \sum_j (x_j-y_j) u_{iij} + t. 
\end{equation*}

Contracting with $F^{ii}$ and employing the identity 
\begin{equation}
    \sum_i F^{ii} u_{iij} = \frac{d}{d x_j} f(x,u,Du) = f_{x_j} + f_z u_j + \sum_k f_{p_k}u_{kj},
\end{equation}
we obtain
\begin{equation}
\label{phi_dd}
\begin{split}
\sum_i F^{ii} \varphi_{ii} &\le  (n-k)f + t \sum_i F^{ii}+ \sum_i f_{p_i}u_{ii}(x_i-y_i)+ C \\
&\leq t \sum_i F^{ii}+ \sum_i f_{p_i}u_{ii}(x_i-y_i)+ C.
\end{split}
\end{equation}

Substituting (\ref{phi_d_square}) and (\ref{phi_dd}) into the maximum principle inequality (\ref{eq_4_8}) and assuming $h \leq 1$, we have
\begin{equation}
\label{max_prin_inequality_1}
\begin{split}
    &h\left(  t \sum_i F^{ii}+\sum_i f_{p_i}u_{ii}(x_i-y_i)+ C \right) \\&\geq \sum_i F^{ii} (x_i-y_i)^2 (u_{ii} + t)^2+h \sum_i (x_i-y_i) u_{ii}f_{p_i}-Ch.
\end{split}
\end{equation}
The terms involving $f_{p_i}$ are identical, so they can be canceled.
\begin{equation}
\label{max_prin_inequality}
\begin{split}
    &h\left( t \sum_i F^{ii}+ C \right) \geq \sum_i F^{ii} (x_i-y_i)^2 (u_{ii} + t)^2.
\end{split}
\end{equation}

We will now demonstrate that if $h$ is chosen to be sufficiently small depending on $r,n, k, \|u\|_{C^{0,1}(B_2(0))}, \left\|\frac{1}{f}\right\|_{L^\infty}$, and $\|f\|_{C^{1,1}}$, then $|x_0-y|<r$, contradicting the fact that $x_0 \in B_\frac{1}{2}(y) \setminus B_r(y)$. We divide the proof into two subcases depending on $k$. 

Subcase $k = n-1$. By Lemma \ref{n-1properties}, we have $F^{ii} = \frac{F^2}{\lambda_i^2}= \frac{f^2}{\lambda_i^2}$ and $\sum_i \frac{f}{\lambda_i} =\sum_i \frac{F}{\lambda_i} = 1$. The maximum principle inequality (\ref{max_prin_inequality}) becomes
\begin{equation*}
     h \left( C + t \sum_i \frac{f^2}{\lambda_i^2} \right) \ge f^2 \sum_i \frac{(x_i-y_i)^2}{\lambda_i^2} (\lambda_i + t)^2. 
\end{equation*}
Because $\sum_i \frac{f^2}{\lambda_i^2} < \left(\sum_i \frac{f}{\lambda_i}\right)^2 = 1$, the left hand side is bounded by $h(C+t)$. 
On the right hand side, since $t > 0$ and $\lambda_i > 0$, we have
\begin{equation*} 
f^2\sum_i (x_i-y_i)^2 \left( 1 + \frac{t}{\lambda_i} \right)^2 \ge f^2 \sum_i (x_i-y_i)^2 = f^2|x-y|^2. 
\end{equation*}
Therefore, $f^2|x-y|^2 \le h(C+t)$. By choosing $h$ sufficiently small depending on $r,n, k, \|u\|_{C^{0,1}(B_2(0))}, \left\|\frac{1}{f}\right\|_{L^\infty}$, we force $|x-y| < r$, which implies the maximum cannot occur in the interior of $\overline{B_\frac{1}{2}(y)} \setminus B_r(y)$. A contradiction arises. 

Subcase $k = n-2$. By Lemma \ref{n-2properties}, the gradient term expands as
\begin{align*}
\sum_i F^{ii} (x_i-y_i)^2 (u_{ii} + t)^2 &= F^2\sum_i \frac{(x_i-y_i)^2}{\lambda_i^2} \left(\sum_{j \neq i}\frac{1}{\lambda_j}\right)  (\lambda_i + t)^2 \\
&= F^2\sum_i (x_i-y_i)^2 \left( \sum_{j \neq i} \frac{1}{\lambda_j} \right) \left( 1 + \frac{t}{\lambda_i} \right)^2.
\end{align*}
The maximum principle inequality (\ref{max_prin_inequality}) thus becomes
\begin{equation}
\label{n-2maxprin}
   h \left(C + t \sum_k F^{kk} \right) \ge F^2\sum_i (x_i-y_i)^2 \left( \sum_{j \neq i} \frac{1}{\lambda_j} \right) \left( 1 + \frac{t}{\lambda_i} \right)^2.
\end{equation}

Claim. For each $i$, $\left( \sum_{j \neq i} \frac{1}{\lambda_j} \right) \left( 1 + \frac{t}{\lambda_i} \right)^2 \ge C \left( \sum_k F^{kk} \right)$.

Subcase 1. $\frac{1}{\lambda_i} \le \frac{1}{2} \sum_j \frac{1}{\lambda_j}$. In this case, we have
\begin{align*}
\sum_{k \neq i} \frac{1}{\lambda_k} &= \left( \sum_k \frac{1}{\lambda_k} \right) - \frac{1}{\lambda_i} \ge \frac{1}{2} \sum_k \frac{1}{\lambda_k}.
\end{align*}
Thus, since $t, \lambda_i > 0$, 
\begin{equation*} 
\left( \sum_{j \neq i} \frac{1}{\lambda_j} \right) \left( 1 + \frac{t}{\lambda_i} \right)^2 \ge \frac{1}{2} \sum_k \frac{1}{\lambda_k} \ge C \sum_k F^{kk},
\end{equation*}
where we use Lemma \ref{n-2_trace_upperbound} in the last inequality.

Subcase 2. $\frac{1}{\lambda_i} > \frac{1}{2} \sum_j \frac{1}{\lambda_j}$. We have
\begin{equation} 
\label{4_9}
\sum_{k \neq i} \frac{1}{\lambda_i \lambda_k} \ge \frac{1}{2} \sum_{k \neq i} \frac{1}{\lambda_k} \sum_j \frac{1}{\lambda_j} \ge \frac{1}{2F},
\end{equation}
where we use the assumption $\frac{1}{\lambda_i} > \frac{1}{2} \sum_j \frac{1}{\lambda_j}$ in the first inequality and the equality $\sum_{j < k} \frac{1}{\lambda_j \lambda_k} = 1/F$ in the second inequality.

Thus,
\begin{align*}
\left( \sum_{k \neq i} \frac{1}{\lambda_k} \right) \left( 1 + \frac{t}{\lambda_i} \right)^2 &\ge \left( \sum_{k \neq i} \frac{1}{\lambda_k} \right) \frac{t^2}{\lambda_i^2} = \frac{t^2}{\lambda_i} \sum_{k \neq i} \frac{1}{\lambda_k \lambda_i} \\
&\ge \frac{t^2}{2F \lambda_i} \ge \frac{t^2}{4F} \sum_j \frac{1}{\lambda_j} \ge C\sum_j\frac{1}{\lambda_j} \ge C \sum_j F^{jj},
\end{align*}
where we use (\ref{4_9}) and the condition $\frac{1}{\lambda_i} > \frac{1}{2} \sum_j \frac{1}{\lambda_j}$ and Lemma \ref{n-2_trace_upperbound}.

Substituting this lower bound into the maximum principle inequality (\ref{n-2maxprin}) and dividing both sides by $\sum_j F^{jj}$, we have
\begin{align*}
h \left( \frac{C}{\sum_j F^{jj}} + t \right) &\ge C \left( \sum_i (x_i-y_i)^2 \right).
\end{align*}
Using the lower bound for the trace $\sum_j F^{jj}$ by Lemma \ref{general_trace_lower_bound}, we obtain
\begin{equation*} 
|x-y|^2 \le Ch.
\end{equation*}
By choosing $h$ sufficiently small depending on $r,n,k, t,\|u\|_{C^{0,1}(B_1(0))},\|\frac{1}{f}\|_{L^\infty}$ and $\|f\|_{C^{1,1}}$, we obtain $|x-y|^2 \le r^2$, ensuring that case 2 cannot happen. 

Case 3. $x_0$ lies on the boundary of $B_{\frac{1}{2}}(y) \backslash B_{r}(y) $. Since $\eta(x) = 0 $ on $\partial B_{1/2}(y)$ and $\eta(x)a_1(x)>0$ if $|x-y|\le R$, we have $x_0 \in \partial B_r(y)$. Hence,
\begin{equation}
\label{doublingb}
    \begin{split}
    \sup_{B_{R}(y)}a_1(x) & \leq \sup_{B_{R}(y)\backslash B_r(y)}a_1(x) + \sup_{B_{r}(y)}a_1(x) \\
    &\leq C  \sup_{B_{R}(y)\backslash B_r(y)}\eta(x)a_1(x) + \sup_{B_{r}(y)}a_1(x) \\
    &\le  C \sup_{B_r(y)}\eta(x)a_1(x) + \sup_{B_{r}(y)}a_1(x)\\
    &\leq C \sup_{B_r(y)} a_1(x) +  \sup_{B_r(y)} a_1(x) \\
    &\leq C  \sup_{B_r(y)} a_1(x). 
    \end{split}
\end{equation}
Hence, either (\ref{doublinga}) or (\ref{doublingb}) holds. We can obtain the desired doubling inequality by rewriting the inequality in terms of $\lambda_1$ and simplifying it.
\end{proof}

Now, we establish the doubling inequality for general $l$.
\begin{proposition}[Doubling inequality]
\label{doubling_general_l}
Let $n \ge 2$, $1 \le k < l \le n$ with $k \in \{l-1, l-2\}$, and let $f(x,u,Du)$ be a positive $C^{1,1}$ function satisfying the structural inequality $\sum_{i,j} f_{p_i p_j} \xi_i \xi_j \ge \frac{K_1}{f} \left(\sum_i f_{p_i} \xi_i\right)^2$ for all $\xi \in \mathbb{R}^n$. Suppose Assumption \ref{assumption1.1} holds and $u \in C^4(B_{2}(0))$ is a strictly convex solution to \eqref{main_equation_2} on $B_2(0)$. Then, for any point $y \in B_1(0)$, there exists a sufficiently small constant $R(n, l, \|u\|_{C^{0,1}(B_1(0))}) > 0$ such that for any $R > r > 0$, we have
\begin{equation}
\begin{split}
     \sup_{B_R(y)} \lambda_1(D^2u)  &\leq C \left(r, n, l, \|u\|_{C^{0,1}(B_2(0))}, \left\|\frac{1}{f}\right\|_{L^\infty}, \|f\|_{C^{1,1}},\sup_{B_r(y)} \lambda_1(D^2u)\right).
\end{split}
\end{equation}
\end{proposition}
\begin{proof}
The proof is highly similar to that of Proposition \ref{doubling}. It suffices to show that case 2 cannot happen. The constant $n-k$ in (\ref{phi_dd}) becomes $l-k$, but it does not affect the inequality. We begin with (\ref{max_prin_inequality}). 

For the case of general $l$ and $k = l - 1$, by equation (3.7) in \cite{LuTsai_general_quotient}, we have 
\begin{equation*}
    1/C \leq F^{ii}(1+u_{ii})^2 \leq C.
\end{equation*}
A simple computation yields
\begin{equation*}
    1/C \leq F^{ii}(t+u_{ii})^2 \leq C.
\end{equation*}
In particular, we have
\begin{equation*}
    \sum_i F^{ii}  \leq \sum_i F^{ii}(t+u_{ii})^2 \leq C.
\end{equation*}
Hence, (\ref{max_prin_inequality}) becomes 
\begin{equation*}
    Ch \geq C|x-y|^2.
\end{equation*}
Choosing $h$ sufficiently small, we show that the maximum point cannot be an interior point.

For the case of general $l$ and $k= l-2$, by equation (3.10) in \cite{LuTsai_general_quotient}, we have
\begin{equation*}
    \frac{\lambda_{l-1}}{C} \le F^{ii}(1+u_{ii})^2 \le C \lambda_{l-1}.
\end{equation*}
A simple computation yields
\begin{equation*}
    \frac{\lambda_{l-1}}{C} \le F^{ii}(t+u_{ii})^2 \le C \lambda_{l-1}.
\end{equation*}
In particular, we have
\begin{equation*}
    C\sum_j F^{jj} \le \sum_j F^{jj}(t+u_{jj})^2 \leq C\lambda_{l-1} \leq  C  F^{ii}(t+u_{ii})^2.
\end{equation*}
Substituting this lower bound for each $i$ into the right hand side of the maximum principle inequality (\ref{max_prin_inequality}), dividing both sides by $\sum_j F^{jj}$, and using the lower bound of $f$ yields
\begin{align*}
h \left( \frac{C}{\sum_j F^{jj}} + t \right) &\ge C \left( \sum_i (x_i-y_i)^2 \right).
\end{align*}
As in the case $k=n-2$, we use the lower bound for the trace $\sum_j F^{jj}$ by Lemma \ref{general_trace_lower_bound}. By choosing $h$ sufficiently small, we show that the maximum point cannot be an interior point.
\end{proof}

\section{Proof of the Main Theorems}
\label{s5}
We follow the framework of Shankar-Yuan \cite{ShankarYuan2024}, and we adapt the arguments from Section 6 of Fan \cite{fan2026hessian}. We give a sketch of the proof. 
\begin{proof}[Proof of Theorem \ref{main_thm}] 
We prove by contradiction.
Let $u_k$ and $f_k$ be sequences of functions that satisfy the assumptions of Theorem \ref{main_thm} with \\ $\|u_k\|_{C^{0,1}(B_2(0))} + \|f_k\|_{C^{1,1}} \leq A$, $\inf f_k \geq f_0 >0$, and $|D^2u_k(0)| \rightarrow \infty$. Choose a subsequence of $u_k$ that converges uniformly in $B_1(0)$ to a continuous function $u \in C^0(B_1(0))$, and choose a subsequence of $f_k$ that converges to $f$ in $C^{1,\alpha}(B_1(0))$ for any $0<\alpha<1$, for some $f \in C^{1,1}(B_1(0))$. 

Step 1. Partial regularity of the limit. Recall that we assume the functions $u_k$ to be convex. Hence, their uniform limit $u$ is convex as well. By the classical Alexandrov theorem, we choose a point $y \in B_{R/2}(0)$, where $R = R(n,\|u\|_{C^{0,1}(B_1(0))})$ is given by Proposition \ref{doubling}, such that $u$ is twice differentiable at $y$. We may choose a quadratic polynomial $Q$ such that
\begin{equation*}
    \sup_{x\in B_r(y)} |u(x)-Q(x)| \leq \sigma(r) = o(r^2).
\end{equation*}

Step 2. Flattening the error and Savin's stability.  We adapt the arguments from Section 6 Step 2 of Fan \cite{fan2026hessian} to rescale the equation and apply the small perturbation theorem. 

To apply the small perturbation theorem, we consider the rescaled functions
\begin{equation*}
    \tilde{v}_k(x) := \dfrac{1}{r^2}\left(u_k(y+r x) - Q(y+r x) \right), \quad x \in B_1(0).
\end{equation*}

Define
\begin{equation*}
    \tilde{f}_k(x,z,p) := f_k(rx+y, r^2z+Q(rx+y), rp+DQ(rx+y)),
\end{equation*}
and the operator
\begin{equation*}
    G(M,p,z,x) := \dfrac{\sigma_n}{\sigma_k}(D^2 Q + M) - \dfrac{\sigma_n}{\sigma_k}(D^2 Q) - \tilde{f}_k(x,z,p) + \tilde{f}_k(x,0,0).
\end{equation*}
By definition, $G(0,0,0,x) \equiv 0$, and a direct calculation shows that the functions $\tilde{v}_k$ satisfy the equation
\begin{equation}
\label{operator_G_def}
    G(D^2\tilde{v}_k, D\tilde{v}_k, \tilde{v}_k, x) = \tilde{f}_k(x,0,0) - \dfrac{\sigma_n}{\sigma_k}(D^2 Q) \quad \text{in } B_1(0).
\end{equation}
By Lemma \ref{lu_eigenvalue_bounds}, $\sigma_k(D^2u_k) \ge C >0$. Hence, $\sigma_k(D^2Q) \ge C $.
Thus, the operator $G(M,p,z,x)$ is uniformly elliptic and continuously differentiable in a neighborhood of the origin. It is also Lipschitz continuous in $p$ and $z$ (since $f,f_k \in C^{1,1}$), and satisfies $G(0,0,0,x) \equiv 0$. 

Furthermore, by the convergence of $f_k \to f$ in $C^{1,\alpha}$ and the continuity of $f$, the right hand side term $\tilde{f}_k(x,0,0) - \frac{\sigma_n}{\sigma_k}(D^2Q)$ can be made arbitrarily small in $C^{0,\alpha}(B_1(0))$ by choosing $r$ small and $k$ large.

Thus, we can choose a sufficiently small $r = r(\sigma)$ and a sufficiently large $k$ such that both $\|\tilde{v}_k\|_{L^{\infty}(B_1(0))}$ and the right hand side error satisfy the smallness conditions $\delta$ of Theorem \ref{generalized_savin}. Applying Theorem \ref{generalized_savin} to $\tilde{v}_k$ and rescaling back to $u_k$, we obtain a uniform bound.
\begin{equation}
\label{small_ball_error}
    \|u_k\|_{C^{2,\alpha}(B_{r/2}(y))} \le C(n, A, f_0, Q, \sigma).
\end{equation}

Step 3. Doubling to propagate the partial regularity.
By Proposition \ref{doubling}, the choice of $r=r(\sigma)$, and \eqref{small_ball_error}, we have
\begin{equation*}
    \begin{split}
        \sup_{B_{3R/4}(y)}|D^2u_k| &\leq C\left(r, n, A, f_0, \sup_{B_{r/2}(y)}|D^2u_k|\right)\\
        &\leq C(r, n, A, f_0, C(n, A, f_0, Q, \sigma))\\
        &\leq C(n, A, f_0, Q, \sigma).
    \end{split}
\end{equation*}
Since $y \in B_{R/2}(0)$, it follows that $0 \in B_{3R/4}(y)$, and therefore we have
\begin{equation*}
    |D^2u_k(0)| \leq C(n, A, f_0, Q, \sigma).
\end{equation*}
This uniformly bounds the Hessians at the origin, contradicting our assumption that $|D^2u_k(0)| \rightarrow \infty$.
\end{proof}
\begin{proof}[Proof of Theorem \ref{main_thm_general_l_k}]
    The proof is nearly identical to the above, except that we use the doubling inequality given by Proposition \ref{doubling_general_l} and apply Lemma \ref{lem:2.3} and Lemma \ref{lem:2.4} to justify that $G$ (defined in (\ref{operator_G_def})) is well defined near the origin.
\end{proof}

\section{Counterexample}
\label{s6}

\begin{proof}[Proof of Remark \ref{thm:counterexample}]
We first prove the case for $l=n$ and $k=n-1$. Consider the sequence of functions
\begin{equation*}
    w_\varepsilon(y) = \varepsilon y_1^2 + y_1^4 + \sum_{i=2}^n \frac{1}{2}y_i^2, \quad \text{for } \varepsilon > 0.
\end{equation*}
These functions are smooth and strictly convex, with their Hessians given by
\begin{equation*}
    D^2w_\varepsilon = \mathrm{diag}(12y_1^2 + 2\varepsilon, 1, \dots, 1),
\end{equation*}
and their Laplacians given by
\begin{equation*}
    \Delta w_\varepsilon = 12y_1^2 + 2\varepsilon + n - 1.
\end{equation*}
Define the right hand side functions $f_\varepsilon(p)$ for $p \in \mathbb{R}^n$ as
\begin{equation*}
    f_\varepsilon(p) = \frac{1}{12p_1^2 + 2\varepsilon + n - 1}.
\end{equation*}
For each $\varepsilon$, let $u_\varepsilon(x)$ be the Legendre transform of the convex function $w_\varepsilon(y)$.
We have
\begin{equation*}
\begin{split}
 x &= Dw_\varepsilon(y) \\
  u_\varepsilon(x) &= \sum_i x_i y_i - w_\varepsilon(y) \\
 y &= Du_\varepsilon(x) \\
D^2u_\varepsilon(x) &= (D^2w_\varepsilon(y))^{-1}.
\end{split}
\end{equation*}
By the relation $x = Dw_\varepsilon(y)$, we have $x_1 = 2\varepsilon y_1+4y_1^3$ and
$x_i = y_i$ for $i =2, \dots, n$.
We may solve for $y_1$ in terms of $x_1$ using Cardano's formula
\begin{equation}
\label{y_formula}
    y_1(x_1) = \sqrt[3]{\frac{x_1}{8}+\sqrt{\frac{x_1^2}{64}+\frac{\varepsilon^3}{216}}} + \sqrt[3]{\frac{x_1}{8}-\sqrt{\frac{x_1^2}{64}+\frac{\varepsilon^3}{216}}}.
\end{equation}
Now, we express $u_\varepsilon(x)$ in terms of $x$.
\begin{equation}
\label{u_formula}
\begin{split}
    u_\varepsilon(x) &= x_1y_1(x_1)-\varepsilon y_1^2(x_1)-y_1^4(x_1) + \frac{1}{2}\sum_{i=2}^n x_i^2 \\
    &= \varepsilon y_1^2(x_1) + 3y_1^4(x_1)+ \frac{1}{2}\sum_{i=2}^n x_i^2 ,
\end{split}
\end{equation}
where we use the relation $x_1 = 2\varepsilon y_1+4y_1^3$ in the last equality. It is clear that $u_\varepsilon$ is well defined and smooth on $B_1(0)$. Observing that $Du_\varepsilon(x) = y$, we have
\begin{equation*}
    \frac{\sigma_n}{\sigma_{n-1}}(D^2u_\varepsilon(x)) = \frac{1}{\Delta w_\varepsilon(y)} = f_\varepsilon(Du_\varepsilon(x)) \quad \text{on } B_1(0).
\end{equation*}
Now, we verify properties (i)-(iii).
For property (i), (\ref{u_formula}) and (\ref{y_formula}), together with the fact that $Du_\varepsilon(x) = y$, yield a uniform $C^{0,1}$ bound for $u_\varepsilon$ on $B_1(0)$.
For property (ii), the functions $f_\varepsilon$ are clearly uniformly bounded in $C^{1,1}$ with respect to the $Du$ variable. Furthermore, assuming $\varepsilon \le 1$, the functions $f_\varepsilon$ satisfy the following uniform lower bound
\begin{equation*}
    f_\varepsilon \ge \frac{1}{ 12(\sup |Du_\varepsilon|)^2+ n +1 } > 0.
\end{equation*}
For property (iii), since $Dw_\varepsilon(0) = 0$, we have 
\begin{equation*}
    D^2u_\varepsilon(0) = (D^2w_\varepsilon(0))^{-1} = \mathrm{diag}\left(\frac{1}{2\varepsilon}, 1, \dots, 1\right).
\end{equation*}
As $\varepsilon \to 0$, the largest eigenvalue $\lambda_1 = \frac{1}{2\varepsilon} \to \infty$.

Now, we generalize this construction to the other pairs of $l$ and $k$ stated in the remark. We use the same definitions of $u_\varepsilon(x)$ and $w_\varepsilon(y)$, but modify the definitions of $f_\varepsilon(Du_\varepsilon(x))$. Using the well known identity
\begin{equation*}
    \sigma_l(\lambda) = \sigma_n(\lambda) \sigma_{n-l}\left(\frac{1}{\lambda}\right),
\end{equation*}
we have 
\begin{equation*}
    \frac{\sigma_l}{\sigma_k}(D^2u_\varepsilon(x)) =\frac{\sigma_{n-l}}{\sigma_{n-k}}(D^2w_\varepsilon(y)).
\end{equation*}
A direct computation yields
\begin{equation*}
    \sigma_m(D^2w_\varepsilon(y)) = (2\varepsilon+12y_1^2) \binom{n-1}{m-1}+\binom{n-1}{m}.
\end{equation*}
The quotient equation can be written as 
\begin{equation*}
    \frac{\sigma_{n-l}}{\sigma_{n-k}}(D^2w_\varepsilon(y)) = \dfrac{(2\varepsilon+12y_1^2) \binom{n-1}{n-l-1}+\binom{n-1}{n-l}}{(2\varepsilon+12y_1^2) \binom{n-1}{n-k-1}+\binom{n-1}{n-k}}=:f_\varepsilon(Du_\varepsilon(x)).
\end{equation*}
By our choice of $l$ and $k$, we have $\binom{n-1}{n-k} > 0$. Therefore, the denominators of the functions $f_\varepsilon$ are uniformly bounded away from $0$. Hence, the functions $f_\varepsilon$ do not blow up and have a uniform $C^{1,1}$ bound in the $Du$ variable. Moreover, for $\varepsilon \le 1$, the functions $f_\varepsilon$ admit a uniform lower bound
\begin{equation*}
    f_\varepsilon(p) \ge \dfrac{\binom{n-1}{n-l}}{(2+12M^2) \binom{n-1}{n-k-1}+\binom{n-1}{n-k}}.
\end{equation*}
Hence, property (ii) is satisfied. Since the construction of $u_\varepsilon$ remains entirely unchanged, properties (i) and (iii) hold exactly as before.
\end{proof}
\begin{remark}
    The proof fails for $k=0$, as $\binom{n-1}{n}=0$ would cause the $C^{1,1}$ norm of $f_\varepsilon$ to blow up as $\varepsilon \rightarrow 0$.
\end{remark}

\section*{Acknowledgements}
The author would like to thank Yi Wang for her valuable guidance and supervision, Jin Yan and Marcin Sroka for helpful comments and pointing out relevant references, and Ravi Shankar for helpful conversations and encouraging discussions regarding the doubling argument.

\printbibliography

@article{Fan2026Savin,
  author  = {Fan, Zhenyu},
  title   = {A generalization of Savin's small perturbation theorem for fully nonlinear elliptic equations and applications},
  journal = {Calculus of Variations and Partial Differential Equations},
  volume  = {65},
  pages   = {182},
  year    = {2026}
}

@article{fan2026hessian,
  author  = {Fan, Zhenyu},
  title   = {Hessian estimates for the sigma-2 equation with variable right-hand side terms in dimension 4},
  journal = {Advances in Mathematics},
  volume  = {494},
  pages   = {110953},
  year    = {2026}
}

@article{WangYu2014SLE,
  author  = {Wang, Dake and Yuan, Yu},
  title   = {Hessian estimates for special Lagrangian equations with critical and supercritical phases in general dimensions},
  journal = {American Journal of Mathematics},
  volume  = {136},
  number  = {2},
  pages   = {481--499},
  year    = {2014}
}

@article{Shankar2024SLE,
  author  = {Shankar, Ravi},
  title   = {Hessian estimates for special Lagrangian equation by doubling},
  journal = {Analysis \& PDE},
  volume  = {19},
  number  = {2},
  pages   = {339--352},
  year    = {2026}
}

@article{LuTsai_general_quotient,
  author  = {Lu, Siyuan and Tsai, Yi-Lin},
  title   = {A note on interior $C^2$ estimate for general Hessian quotient equation},
  journal = {Communications on Pure and Applied Analysis},
  volume  = {33},
  pages   = {88--100},
  year    = {2026}
}

@article{Lu2025_top_quotient,
  author  = {Lu, Siyuan},
  title   = {Interior $C^2$ estimate for Hessian quotient equation in general dimension},
  journal = {Annals of PDE},
  volume  = {11},
  number  = {2},
  pages   = {Paper No. 17, 26 pp.},
  year    = {2025}
}

@article{Bombieri1969,
  author  = {Bombieri, Enrico and De Giorgi, Ennio and Miranda, Mario},
  title   = {Una maggiorazione a priori relativa alle ipersuperfici minimali nonparametriche},
  journal = {Archive for Rational Mechanics and Analysis},
  volume  = {32},
  pages   = {255--267},
  year    = {1969}
}

@article{Trudinger1972,
  author  = {Trudinger, Neil S.},
  title   = {A new proof of the interior gradient bound for the minimal surface equation in $n$ dimensions},
  journal = {Proceedings of the National Academy of Sciences of the United States of America},
  volume  = {69},
  number  = {4},
  pages   = {821--823},
  year    = {1972}
}

@article{Korevaar1987,
  author  = {Korevaar, Nicholas J.},
  title   = {A priori interior gradient bounds for solutions to elliptic Weingarten equations},
  journal = {Annales de l'Institut Henri Poincar{\'e} C, Analyse non lin{\'e}aire},
  volume  = {4},
  number  = {5},
  pages   = {405--421},
  year    = {1987}
}

@article{Zhou2021,
  author  = {Zhou, Xingchen},
  title   = {Hessian estimates to special Lagrangian equation on general phases with constraints},
  journal = {Calculus of Variations and Partial Differential Equations},
  volume  = {61},
  number  = {1},
  pages   = {4},
  year    = {2022}
}

@article{Chen2009,
  author  = {Chen, Jingyi and Warren, Micah and Yuan, Yu},
  title   = {A priori estimate for convex solutions to special Lagrangian equations and its application},
  journal = {Communications on Pure and Applied Mathematics},
  volume  = {62},
  number  = {4},
  pages   = {583--595},
  year    = {2009}
}

@article{MichaelSimon1973,
  author  = {Michael, Joseph H. and Simon, Leon M.},
  title   = {Sobolev and mean-value inequalities on generalized submanifolds of $\mathbb{R}^n$},
  journal = {Communications on Pure and Applied Mathematics},
  volume  = {26},
  number  = {3},
  pages   = {361--379},
  year    = {1973}
}

@article{ShankarYuansigma2dim4,
  author    = {Shankar, Ravi and Yuan, Yu},
  title     = {Hessian estimates for the sigma-2 equation in dimension four},
  journal   = {Annals of Mathematics},
  volume    = {201},
  number    = {2},
  pages     = {489--513},
  year      = {2025},
  publisher = {Department of Mathematics of Princeton University}
}

@article{Qiu2024,
  author  = {Qiu, Guohuan},
  title   = {Interior Hessian Estimates for $\sigma_2$ Equations in Dimension Three},
  journal = {Frontiers of Mathematics},
  volume  = {19},
  pages   = {577--598},
  year    = {2024}
}

@article{Savin2007,
  author  = {Savin, Ovidiu},
  title   = {Small perturbation solutions for elliptic equations},
  journal = {Communications in Partial Differential Equations},
  volume  = {32},
  number  = {4},
  pages   = {557--578},
  year    = {2007}
}

@article{GuanQiu2019,
  author    = {Guan, Pengfei and Qiu, Guohuan},
  title     = {Interior {$C^2$} regularity of convex solutions to prescribing scalar curvature equations},
  journal   = {Duke Mathematical Journal},
  volume    = {168},
  number    = {9},
  pages     = {1641--1663},
  year      = {2019},
  publisher = {Duke University Press}
}

@article{ShankarYuan2024,
  author  = {Shankar, Ravi and Yuan, Yu},
  title   = {Regularity for the {M}onge--{A}mp{\`e}re equation by doubling},
  journal = {Mathematische Zeitschrift},
  volume  = {307},
  number  = {2},
  pages   = {34},
  year    = {2024}
}

@article{fung,
  author  = {Fung, Cheuk Yan},
  title   = {Doubling argument of the Hessian estimate for the special Lagrangian equation on general phases with constraints},
  journal = {Journal of Differential Equations},
  volume  = {482},
  pages   = {114687},
  year    = {2026},
}

@article{vis_lemma,
  author    = {Brendle, Simon and Choi, Kyeongsu and Daskalopoulos, Panagiota},
  title     = {Asymptotic behavior of flows by powers of the {G}aussian curvature},
  journal   = {Acta Mathematica},
  volume    = {219},
  number    = {1},
  pages     = {1--16},
  year      = {2017},
  publisher = {International Press of Boston}
}

@article{Zhou2024,
  title={Notes on generalized special Lagrangian equation},
  author={Zhou, Xingchen},
  journal={Calculus of Variations and Partial Differential Equations},
  volume={63},
  number={8},
  pages={197},
  year={2024}
}

@article{JiaoSui2026,
  title={Interior Hessian estimates for Hessian quotient equations in dimension three},
  author={Jiao, Heming and Sui, Zhenan},
  journal={preprint arXiv:2602.14064},
  year={2026}
}

@article{LuSroka2026,
  title={On Liouville's theorem for the Hessian quotient equation $\sigma_2/\sigma_1$},
  author={Lu, Siyuan and Sroka, Marcin},
  journal={preprint arXiv:2602.14946},
  year={2026}
}

@article{MeiYan2026,
  title={Interior $C^2$ estimate for semi-convex solutions to a class of Hessian quotient equations in arbitrary dimensions},
  author={Mei, Xinqun and Yan, Jin},
  journal={preprint arXiv:2604.23349},
  year={2026}
}

@article{BSW25Lagrangian_Largangian_mean_curvature,
  title={Doubling and the two-dimensional critical valued Lagrangian phase},
  author={Bhattacharya, Arunima and Shankar, Ravi and Wall, Jeremy},
  journal={arXiv preprint arXiv:2510.22887},
  year={2025}
}

@article{FanShankar_scalar_curvature_equation,
  title={Interior estimates and regularity for the scalar curvature equation in dimension 4},
  author={Fan, Zhenyu and Shankar, Ravi},
  journal={arXiv preprint arXiv:2608.22748},
  year={2026}
}

@article{Shankar2025Singularities,
  title={Removing singularities for fully nonlinear {PDE}s},
  author={Shankar, Ravi},
  journal={Proceedings of the American Mathematical Society},
  volume={153},
  number={11},
  pages={4743--4751},
  year={2025}
}

@article{Tsai2026concavity,
  author        = {Yi-Lin Tsai},
  title         = {A concavity inequality for {H}essian quotient equations},
  journal       = {arXiv preprint arXiv:2608.16383},
  year          = {2026}
}

@article{LiWu2026concavity,
  author        = {Zhisu Li and Ke Wu},
  title         = {A concavity inequality and interior {$C^2$} estimate for {H}essian quotient equations},
  journal       = {arXiv preprint arXiv:2608.17405},
  year          = {2026}
}

@article{DongZhang2026concavity,
  author        = {Weisong Dong and Ruijia Zhang},
  title         = {Interior {H}essian estimates for {H}essian quotient equations},
  journal       = {arXiv preprint arXiv:2608.19087},
  year          = {2026}
}

@article{GuanSroka2026,
  title={A special concavity property for positive Hessian quotient operators},
  author={Guan, Pengfei and Sroka, Marcin},
  journal={Discrete and Continuous Dynamical Systems},
  volume={54},
  pages={50--60},
  year={2026}
}
\end{document}